\documentclass{amsart}

\usepackage{latexsym,mathrsfs,bm}
\usepackage[english]{babel}
\usepackage{paralist}
\usepackage{amsthm,amsfonts,amssymb,amsmath}
\usepackage{graphicx}

\usepackage[latin1]{inputenc} 
\newcommand{\eps}{\varepsilon}
\newcommand{\R}{\mathbb{R}}

\newcommand{\C}{\mathbb{C}}

\newcommand{\Z}{\mathbb{Z}}

\newcommand{\es}[1]{\begin{equation}\begin{split}#1\end{split}\end{equation}}
\newcommand{\est}[1]{\begin{equation*}\begin{split}#1\end{split}\end{equation*}}

\newcommand{\s}{\operatorname{\mathcal{S}}}
\newcommand{\sss}{\operatorname{s}}

\newcommand{\tn}[1]{\textnormal{#1}}

\renewcommand{\mod}[1]{~\pr{\textnormal{mod}~#1}}
\newtheorem*{theo*}{Theorem}
\newtheorem{theo}{Theorem}

\newtheorem{ezer}{Exercise}

\newtheorem{prop}[ezer]{Proposition}

\newtheorem{lemma}{Lemma}
\newtheorem{corol}[lemma]{Corollary}
\newtheorem{remark}{Remark}

\newtheorem*{rem*}{Remark}
\def\sumstar{\operatornamewithlimits{\sum\nolimits^*}}
\newcommand{\pr}[1]{\left( #1\right)}
\newcommand{\pg}[1]{\left\{ #1\right\}}
\newcommand{\pmd}[1]{\left| #1\right|}

\newcommand{\sgn}{\operatorname{sgn}}
\newcommand{\e}[1]{\operatorname{e}\pr{ #1}}
\newcommand{\cc}{\operatorname{c}}

\let\originalleft\left
\let\originalright\right
\renewcommand{\left}{\mathopen{}\mathclose\bgroup\originalleft}
\renewcommand{\right}{\aftergroup\egroup\originalright}

\newcommand{\comment}[1]{}

\numberwithin{equation}{section}
\begin{document}

\title[On the reciprocity law for twisted Dirichlet $L$-functions]{On the reciprocity law for the twisted second moment of Dirichlet $L$-functions}
\author {Sandro Bettin}
\address{  Sandro Bettin\\
Centre de Recherches Math\'{e}matiques - Universit\'{e} de Montr\'{e}al, P.O. Box 6128, Centre-ville Station, Montr\'{e}al, QC, H3C 3J7, Canada
}
\curraddr{}
\email{bettin@crm.umontreal.ca}
\thanks{}


\begin{abstract}
We investigate the reciprocity law, studied by Conrey~\cite{Con07} and Young~\cite{You11a}, for the second moment of Dirichlet L-functions twisted by $\chi(a)$ modulo a prime $q$. We show that the error term in this reciprocity law can be extended to a continuous function of $a/q$ with respect to the real topology. Furthermore, we extend this reciprocity result, proving an exact formula involving also shifted moments.

We also give an expression for the twisted second moment involving the coefficients of the continued fraction expansion of $a/q$, and, consequently, we improve upon a classical result of Selberg on the second moment of Dirichlet L-functions with two twists.

Finally, we obtain a formula connecting the shifted second moment of the Dirichlet $L$-functions with the Estermann function. In particular cases, this result can be used to obtain some simple explicit exact formulae for the moments.

\end{abstract}

\maketitle

\section{Introduction}
Since the work of Hardy and Littlewood~\cite{HL}, the study of mean-values of $L$-functions has played a central role in analytic number theory. This is due to the number of direct applications on several classical problems on $L$-functions, such as on their maximum size, on the proportion of zeros satisfying the Riemann hypothesis, on non-vanishing at the central points, and on Siegel zeros (see, for example,~\cite{Sou08,Lev,Sou00,IS}).

For these applications, one typically needs to understand twisted moments, which can be used to amplify large values or to  ``mollify'' the $L$-functions. Moreover, twisted moments indicate more clearly the structure and the symmetries of the moments. 

In~\cite{Con07}, Conrey considered the twisted second moment of Dirichlet $L$-functions at the central point,
\es{\label{ccrf}
M\pr{a,q}:=\frac{q^\frac12}{\varphi(q)}\sumstar_{\chi \mod q}\pmd{L\pr{\tfrac12,\chi}}^2\chi(a),
}
where $\sum^*$ indicates that the sum is restricted to primitive characters. (Notice that $M\pr{a,q}$ is real.) Conrey computed the asymptotic for $M\pr{a,q}$ and observed that, when $a,q$ are primes, $M\pr{a,q}$ satisfies an approximate reciprocity relation, highlighting a symmetry which is not immediately visible from the definition. More precisely, Conrey showed that for primes $a,q$ such that $2\leq a< q$, one has
\est{
M\pr{a,q}-M\pr{-q,a}=\frac{q^\frac12}{a^\frac12}\pr{\log\frac qa+\gamma-\log8\pi}+\zeta\pr{\tfrac12}^2 +O\pr{\frac a{q^\frac12}+q^{-\frac12}\log q+\frac{1}{a^\frac12}\log q}.
}
This formula gives an asymptotic formula for $M\pr{a,q}-M\pr{-q,a}$ as long as $a=o(q^{\frac23})$. For comparison notice that the asymptotic formula for $M\pr{a,q}$ is known only on the smaller range $a=o( q^{\frac12})$, in which case we have
\es{\label{afcr1}
M(a,q)\sim (q/a)^{\frac12}\log (q/a). \\
}

In~\cite{You11a}, Young gave a new and more direct proof of~\eqref{ccrf}, improving also the error term. A slight reformulation of his result states that, for primes $a,q$ such that $2\leq a<q$, one has
\es{\label{yf}
M\pr{a,q}- M\pr{-q,a}&=\frac{q^\frac12}{a^\frac12}\pr{\log\frac qa+\gamma-\log8\pi}+\\
&\quad+\zeta\pr{\tfrac12}^2\pr{1-2\frac{q^\frac12}{\varphi(q)}(1-q^{-\frac12})+2\frac{a^\frac12}{\varphi(a)}(1-a^{-\frac12})} +\mathcal {E}(a,q),
}
where
\est{
\mathcal {E}(a,q)\ll aq^{-1+\eps}+ a^{-C},
}
for all fixed $\eps,C>0$. Notice, in particular, that~\eqref{yf} gives the asymptotic for $M\pr{a,q}-M\pr{-q,a}$, when $a\ll q^{1-\eps}$.

\begin{figure}[h]
\centering
\includegraphics[width=0.850\textwidth]{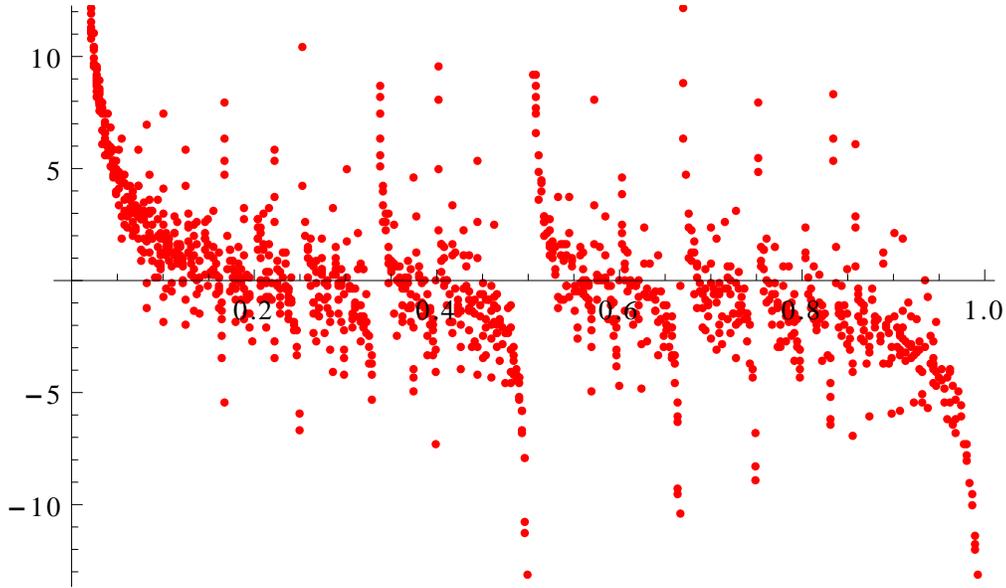}
\caption{$(a/q,M(a,q))$ for primes $a,q$ with $a<q\leq 229$.}
\label{fig:M_0}
\end{figure}
If one graphs the error term $\mathcal {E}(a,q)$ as a function $\mathcal {E}(\frac aq)$ of rational numbers $0<\frac aq<1$ with prime numerators and denominators, one sees that $\mathcal {E}(x)$ is not a ``chaotic'' error term and one is led to guess that it is extendable to a continuous function of $x\in[0,1]$. (See Figure~\ref{fig:cont} below). This first impression is indeed correct.
\begin{theo}\label{et}
Let $a,q\geq2$ be primes with $a\neq q$. Then, $\mathcal {E}(\frac aq):=\mathcal {E}\pr{a,q}$ extends to a continuous function $\mathcal {E}\pr{x}$ of the non-negative real numbers, which is $O(x)$ as $x\rightarrow0^+$. In particular $\mathcal {E}(\frac aq)\ll a/q$ for $a\ll q$.
\end{theo}

In fact, one can say quite more, but first, for convenience of notation, we define the shifted twisted second moment with a different weight attached to the $L$-function associated to the principal character (i.e., essentially to $\zeta$):
 \est{
M^*_j\pr{a,q}&:=\frac{q^{\frac12-j}}{\varphi(q)}\hspace{-0.168cm}\sumstar_{\chi\mod q} L\pr{\tfrac12-j,\overline\chi}L\pr{\tfrac12+j,\chi}\chi(a)+\frac{q^{-j}}{\varphi(q)} \zeta\pr{\tfrac12+j}\zeta\pr{\tfrac12-j}(2q^\frac12 -q^j-q^{-j})\\
&=\frac{q^{\frac12-j}}{\varphi(q)} \sum_{\chi\mod q} L\pr{\tfrac12-j,\overline\chi}L\pr{\tfrac12+j,\chi}\chi(a)+\frac{q^{-j}}{\varphi(q)} \zeta\pr{\tfrac12+j}\zeta\pr{\tfrac12-j}(q^\frac12 -q^{-\frac12}).
}
We also define $P_j(x)$ as the degree $j$ polynomial given by
\est{
P_j\pr{x}:=\sum_{\ell=0}^j\binom{j-\frac12}{\ell-\frac12}\zeta\pr{\tfrac12+\ell}\zeta\pr{\tfrac12-\ell}x^{\ell}.
}

\begin{theo}\label{mt}
Let $a,q>0$ be different primes. Then
\es{\label{mtf}
M_0^*\pr{\pm a,q}&=\sum_{j=0}^\infty \binom{j-\frac12}{j}\bigg(\pr{\frac{\mp a}{q}}^jM_j^*\pr{\mp q,a}+r_{\pm,j}(\tfrac aq)-P_j\pr{\tfrac{\mp a}{q}}\bigg)+W_{\pm}(\tfrac aq)+g_{\pm}(\tfrac aq)-r_{\mp,0}(\tfrac qa),\\
}
where 
\est{
W_{\pm}(x)&:=\frac1{2\pi i}\int_{\pr{-\frac12}}\frac{\Gamma(w)}{ \sin \pi w} \zeta\pr{\frac12+w}^2 \pr{\cos\pr{\tfrac \pi 2w}\pm \sin\pr{\tfrac \pi 2w}}(2\pi x)^{-w}\,dw,\\
g_{\pm}(x)&:=\zeta(\tfrac 12)^2\pr{\tfrac12\pm\tfrac12-\tfrac1\pi(\log (z/4)},\\
r_{\pm,j}(z)&:=\begin{cases}
\frac{\pi }{2}z^{\frac12} & \tn{if }\pm=+,\\
(\log( 2\pi/z)-\Psi(\tfrac12-j)-2\gamma)z^{\frac12} & \tn{if }\pm=-.
\end{cases}
}
As usual, $\int_{(c)}\cdot\ ds$ indicates that the integral is taken along the vertical line from $c-i\infty$ to $c+i\infty$ and $\Psi(x)$ is the digamma function.
\end{theo}

\begin{remark}
Notice that $W_\pm(x)$ is analytic on $\C\setminus\R_{\leq 0}$ and, by contour integration, it satisfies the asymptotic expansion
\est{
W_\pm(x)=\sum_{j=1}^N (c_{1,\pm}\log x +c_ {2,\pm})x^j+O_{N,\eps}\pr{|x|^{N+1}\log |2+x|},
}
in $|\arg x|<\pi-\eps$, as $x\rightarrow0$, for all $N\geq1$, $\eps>0$.
\end{remark}

\begin{remark}
It is easy to show that Theorem~\ref{mt} (as well as Theorem~\ref{mtc} below) holds also if $a=1$ or $q=1$. In particular, taking $a=1$, one obtains an exact formula for the second moment of the Dirichlet $L$-functions.
\end{remark}

The sum on the right hand side of~\eqref{mtf} is not uniformly convergent on $\frac aq\in \R_{\geq 0}$ and thus we can not immediately deduce Theorem~\ref{et}. This is essentially due to the fact that Gauss' hypergeometric formula (one of the crucial tools in the proof of~\eqref{mtf}), 
\es{\label{ghf}
\frac{\Gamma(c)\Gamma(c-a-b)}{\Gamma(c-a)\Gamma(c-b)}={}_2F_1 (a,b;c;1),
}
holds only for $\Re(a+b)<\Re(c)$. As usual ${}_2F_1 (a,b;c;1)$ indicates the hypergeometric function
\est{
{}_2F_1 (a,b;c;1)=1+\sum_{n=1}^\infty \frac{a(a+1)\cdots(a+n-1)\cdot b(b+1)\cdots(b+n-1)}{c(c+1)\cdots(c+n-1)} \frac{z^n}{n!}.
}
The following theorem is obtained by replacing~\eqref{ghf} with an approximation valid for all $a,b,c$, obtained with a strategy close in spirit to that of the main Lemma in the recent works of Kaczorowski and Perelli on the Selberg class (see Lemma A of~\cite{KP02}).

\begin{theo}\label{mtc}
Let $a,q>0$ be different primes. Let $N\geq0$. Then
\es{\label{mtfa}
M_0^*\pr{\pm a,q}&=\sum_{j=0}^{2N} \binom{j-\frac12}{j}\bigg(\pr{\frac{\mp a}{q}}^jM_j^*\pr{\mp q,a}+\frac{\zeta(\frac12+j)\zeta(\frac12-j)}{j!}\pr{\frac{\mp a}{q}}^j\bigg)-r_{\mp}(q/a)+\psi_N(\tfrac{\pm a}{q}),\\
}
where $r_{\pm}(x)$ is as in Theorem~\ref{mt} and $\psi_N(x)$ is a $\mathcal C^N$ function on $\R$ satisfying $\psi_N(x)\ll x^{2N+1}$ for all $\eps>0$, $x\ll 1$. 
\end{theo}

\begin{remark}
One can also truncate at $j=N$ the sum on the right of~\eqref{mtfa}, adding the terms $j>N$ to $\psi_N$ and getting
\est{
M_0^*\pr{\pm a,q}&=\sum_{j=0}^{N} \binom{j-\frac12}{j}\bigg(\pr{\frac{\mp a}{q}}^jM_j^*\pr{\mp q,a}+\frac{\zeta(\frac12+j)\zeta(\frac12-j)}{j!}\pr{\frac{\mp a}{q}}^j\bigg)-r_{\mp}(q/a)+\tilde\psi_N(\tfrac{\pm a}{q}).\\
}
With this choice, the error function $\tilde \psi_N$ satisfies  $\tilde\psi_N(x)\ll x^{N+1}$, for $x\ll1$, and is $\mathcal C^N$ in $\R\setminus\{0\}$ and $C^{[N/2]}$ at $x=0$ (cf. Remark~\ref{rk2} below). 
\end{remark}

Notice that Theorem~\ref{mtc} shows that the discontinuities (with respect to the real topology) of $M_0^*(a,q)$ as a function of rational numbers $\frac aq$ can be removed by subtracting the reciprocal shifted twisted moments, $(-1)^j{{j-\frac12}\choose{j}}\frac{a^j}{q^j}M^*_j(-q,a)$. See Figure~\ref{fig:cont} below. 

\begin{figure}[h]
\includegraphics[width=0.485\textwidth]{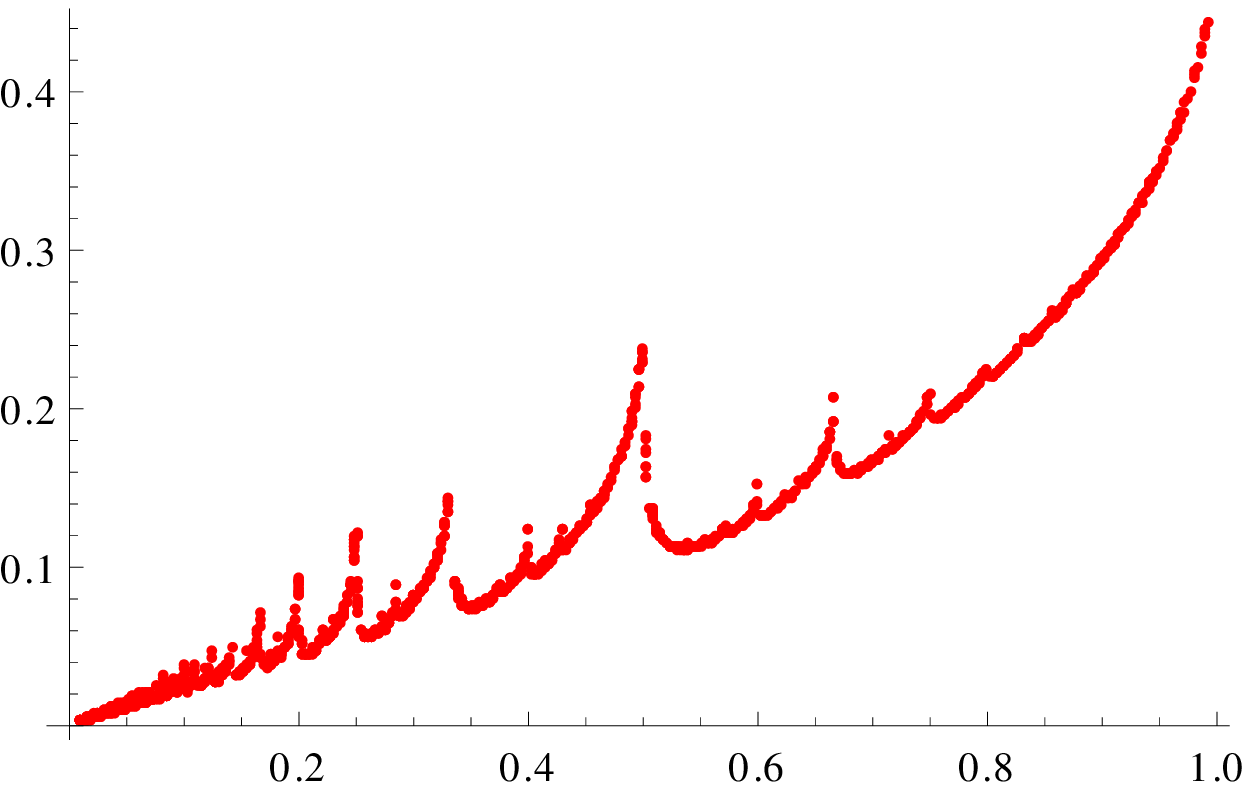}\vspace{0.3cm}
\includegraphics[width=0.485\textwidth]{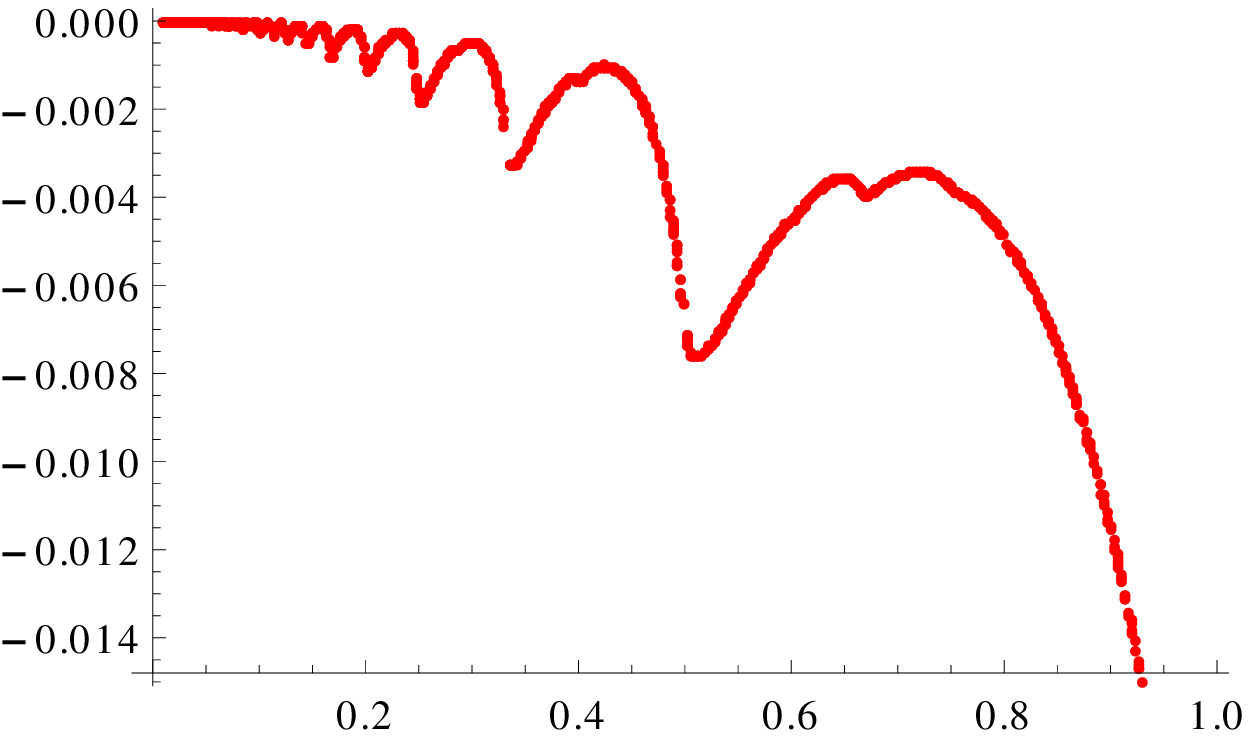}\hspace{0.42cm}
\includegraphics[width=0.485\textwidth]{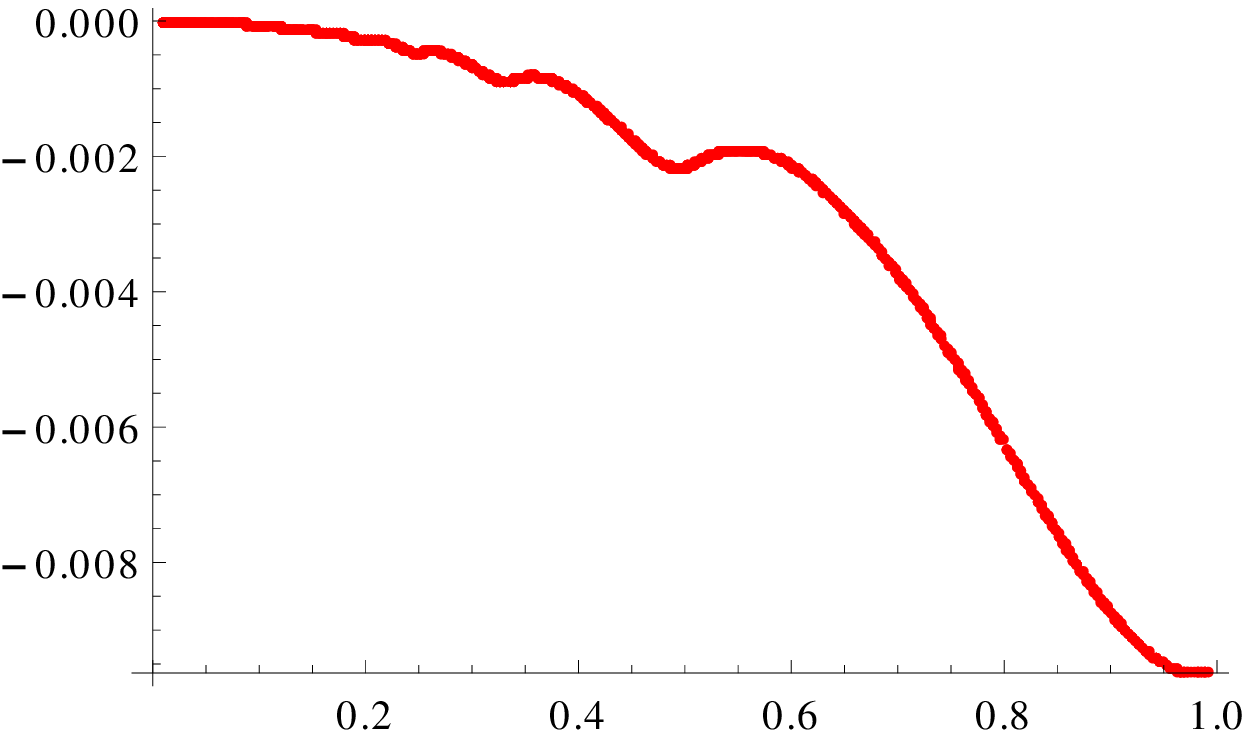}
\includegraphics[width=0.485\textwidth]{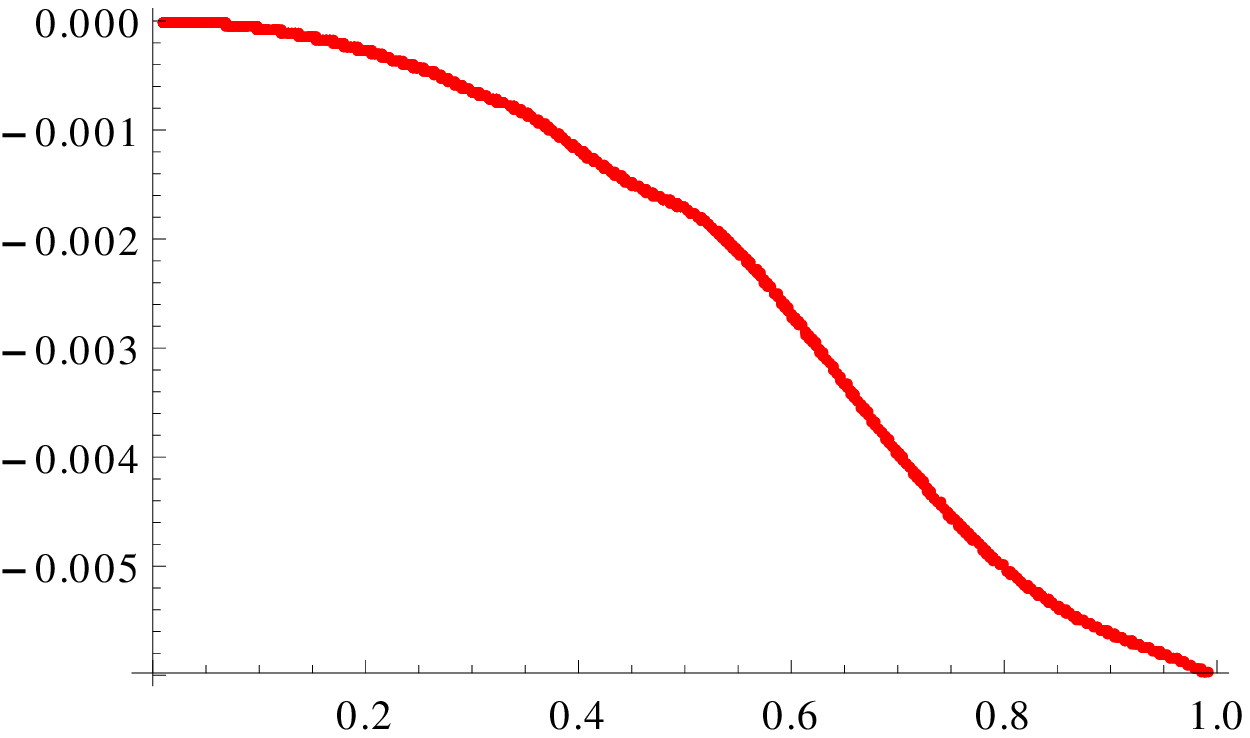}
\caption{A graph of $\tilde \psi_N(\frac aq)$ (notice that $\tilde\psi_0(a/q)$ coincide with $\mathcal E(a/q):=\mathcal E(a,q)$, defined by~\eqref{yf}) for $N=0,1,2,3$, where $a,q$ vary among primes $a<q\leq 229$ and larger primes with $a/q$ close to rationals with small denominators.}
\label{fig:cont}
\end{figure}

Theorem~\ref{mt} should be compared with the main result of~\cite{BC13a}. In that paper Conrey and the author proved that a cotangent sum $c_0\pr{\frac hk}$, strictly related to the the twisted second moment of the Riemann zeta-function, satisfies a reciprocity relation $c_0\pr{\frac hk}+\frac khc_0\pr{\frac hk}-\frac1{\pi h}=\psi\pr{\frac hk}$, where $\psi(x)$ is analytic on $\C\setminus\R_{\leq 0}$. Moreover, in~\cite{Bui}, Bui gives an approximate reciprocity relation for the twisted second moment of $L$-functions associated with primitive Hecke eigenforms of weight 2. He does this by computing separately the asymptotic for the moments appearing in the formula. It would be interesting to see whether there is a  more direct proof of Bui's theorem, which extends his result to a wider range, giving a more genuine reciprocity relation also for the second moment of $L$-functions in this family. More generally, one might also speculate whether all twisted moments of $L$-functions satisfy some hidden reciprocity formula.

In his paper, Young observes also that ``the reciprocity relation is not self-dual, so it could potentially be used recursively to obtain a curious kind of asymptotic expansion''. (The limitation coming from the requirement of the primality of $a,q$ and of all the integers encountered in this recursion can be superseded by proving a reciprocity formula for an intermediate function, valid for all integers). Following Young's observation, one arrives to the following theorem.
\begin{theo}\label{Ypo}
Let $a,q\in\Z_{>0}$ with $q$ prime. Let $[b_0;b_1,\cdots,b_{\kappa}]$ be the continued fraction expansion of $a/q$ and let $v_j$ be the $j$-th partial denominator. Then
 \es{\label{prefo1t}
M_0^*\pr{\pm a,q}&= \sum_{\substack{j=1,\\ (-1)^j=\mp1}}^{\kappa}\pr{\frac {v_{j}}{v_{j-1}}}^{\frac12}\pr{\log \frac {v_{j}}{v_{j-1}}+\gamma -\log 8\pi }-\frac {\pi }2\sum_{\substack{j=1,\\ (-1)^j=\pm1}}^{\kappa}\pr{\frac {v_{j}}{v_{j-1}}}^{\frac12}+{}\\
&\quad+\zeta\pr{\tfrac12}^2({\kappa}+1)+\sum_{j=1}^{\kappa} \psi_0\pr{\pm  (-1)^{j}\frac {v_{j-1}}{v_{j}}},
}
where $\psi_0$ is as in Theorem~\ref{mtc}.
\end{theo}

\begin{corol}\label{Ypc}
Let $a,q\in\Z_{>0}$ with $q$ prime. Let $[b_0;b_1,\cdots,b_{\kappa}]$ be the continued fraction expansion of $\frac aq$. Then
\es{\label{prr}
M(a,q)=\sum_{\substack{j=1,\\j\tn{ odd}}}^{\kappa}b_j^{\frac12}(\log b_j +\gamma -\log 8\pi)-\frac\pi 2\sum_{\substack{j=1,\\j\tn{ even}}}^{\kappa}b_j^{\frac12}+O({\kappa})\\
}
and
\est{
&\frac{q^{\frac12}}{\varphi(q)}\sumstar_{\substack{\chi\mod q,\\\chi(-1)=\pm1}} \pmd{L\pr{\frac12,\chi}}^2\chi(a)=\pm \frac12\sum_{j=1}^{\kappa}(\pm 1)^jb_j^{\frac12}(\log b_j +\gamma -\log 8\pi \mp\frac {\pi }2)+O({\kappa}).\\
}
\end{corol}

Corollary~\ref{Ypc}, which should be seen as the natural generalization of~\eqref{afcr1}, is particularly interesting as it identifies exactly the values of $a$ for which the twisted moments $M(a,q)$ are very large. Indeed these large moments (e.g. satisfying $|M(a,q)|\gg q^{\frac14+\eps}$) correspond to the values of $a$ for which the continued fraction expansion of $\frac aq$ contains a large coefficient (e.g. $\frac aq= [0;b_1,\dots,b_{\kappa}]$ with some $b_i \gg q^{\frac12+\eps}$) and thus to the $\frac aq$ which are ``very close'' to a rational number with ``small'' denominator.

Corollary~\ref{Ypc} can also be used to investigate the second moment of Dirichlet L-functions with two twists:
\est{
M_\pm (h,k;q):=\frac{q^{\frac12}}{\varphi(q)}\sumstar_{\substack{\chi\mod q,\\\chi(-1)=\pm1}} \pmd{L\pr{\frac12,\chi}}^2\chi(h)\overline\chi(k).
} 
The problem of finding an asymptotic for $M_\pm (h,k;q)$ was first considered by Selberg~\cite{Sel}, who obtained the asymptotic formula
\es{\label{afhk}
M_\pm (h,k;q)\sim \frac12 \pr{\frac{q}{hk}}^\frac12\pr{\log \frac q{hk} +\gamma -\log 8\pi \mp\frac {\pi }2}
}
 in the case $hk\max(h^2,k^2)=o(q\log^2 q)$ (and $q$ prime). Iwaniec and Sarnak~\cite{IS} considered the same problem in their paper on non-vanishing of the central value of Dirichlet L-functions~\cite{IS}, showing that the asymptotic formula~\eqref{afhk} holds on average for $h,k\ll q^{\frac12-\eps}$. 

Using Corollary~\ref{Ypc} (and a simple expression for the continued fraction expansion of $\frac {h\overline k}q$, cf. Lemma~\ref{hat} below), we are able to improve upon the result of Selberg, extending the asymptotic formula~\eqref{afhk} to the range $hk\max(h,k)=o(q)$. 

\begin{corol}\label{csb}
Let $q$ be a prime and let $1\leq h,k<q$ with $(h,k)=1$. Then, if $q\geq 4hk$ we have
\es{\label{csbf}
M_\pm (h,k;q)= \frac12 \pr{\frac{q}{hk}}^\frac12\pr{\log \frac q{hk} +\gamma -\log 8\pi \mp\frac {\pi }2}+O((h+k)^{\frac12}\log q).
} 
\end{corol}

In the case when $h,k$ and $q$ are all primes Theorem~\ref{Ypo} implies also a $3$-terms relation.
\begin{corol}\label{c3t}
Let $h,k,q$ be different primes and let $q\geq 4hk$. Then
\est{
M_\pm(h, k,q)
&=\pm M_\pm(h, q;k)\pm M_\pm(k, q;h)+\frac12\pr{\frac q{hk}}^{\frac12}(\log \frac{q}{hk} +\gamma -\log 8\pi\mp\frac {\pi }2)+O(\log  q).\\
}
\end{corol}

\begin{remark}
Corollary~\ref{c3t} indicates clearly that the condition $hk\max(h,k)=o( q)$ is necessary for the asymptotic formula~\eqref{afhk} to hold. More precisely, one can show that~\eqref{afhk} doesn't hold on the range $h\asymp k\asymp q^{\frac13}$ (on average, however, one expects that this asymptotic formula holds true on the wider range $h,k\ll q^{1-\eps}$). Indeed, if we take $h,k,q$ to be primes with 
\es{\label{afv}
k < U h <2k,\qquad q^{\frac13}<k<2q^{\frac13},\qquad q\equiv h\mod k,
}
where $U$ is a sufficiently large constant, then Corollary~\ref{c3t} implies
\es{\label{hg}
M_\pm(h, k,q)
&=\frac {k^\frac12}2\log k+\frac12\pr{\frac q{hk}}^{\frac12}\log \frac{q}{hk}+O(h\log h+Uq^{\frac13})\geq\frac12(1+\frac 1U)\pr{\frac q{hk}}^{\frac12}\log \frac{q}{hk},\\
}
for $q$ large enough. (The existence of arbitrary large primes $h,k,q$ satisfying~\eqref{afv} is not obvious, but from the proof of Corollary~\ref{hg} one easily sees that~\eqref{hg} holds also without the assumption of the primality of $h,k$, and so the goal of finding suitable large integers $h,k$ and primes $q$ becomes easily fulfillable.)
\end{remark}

We remark also that periodic functions $f(\frac aq)$ which admit simple expressions in terms of the continued fraction expansion of $\frac aq$ (such as $M(a,q)$, by Corollary~\ref{Ypc}) are strictly related to additive functions of $SL(2,\Z)$, modulo the parabolic elements. In recent years, there have been several works studying the distribution of functions of this kind, see, for example,~\cite{Var,PR,Moz}.  It is possible that a similar approach might work also in the study of the distribution of $M\pr{a,q}$. (Anyway, in a forthcoming work we choose a different route, computing all the moments for $M\pr{a,q}$ by using classical methods of moments of $L$-functions). This would be especially interesting because it would also give a new approach to the $4$-th moment of Dirichlet $L$-functions at the central point. Indeed, by the orthogonality of Dirichlet characters one has that the second moment of $M\pr{a,q}$ is 
\est{
\frac 1q \sum_{a=1}^q M\pr{a,q}^2=\frac 1{\varphi(q)}\sumstar_{\chi \mod q}\pmd{L\pr{\frac12,\chi}}^4.
}
Young~\cite{You11b} gave an asymptotic for the right hand side, combining different methods to handle certain averages of Kloosterman sums. It would be nice to see whether one can give an alternative proof of his result using Theorem~\ref{Ypo}. For the moment, we content ourselves to use Young's result in the opposite direction, proving the following Corollary.
\begin{corol}\label{fccc}
Let $q$ be a prime. For $1\leq a\leq q$, let
\est{
f_{\pm}(a/q)&:= \sum_{j=1}^{\kappa}(\pm 1)^jb_j^{\frac12}(\log b_j +\gamma -\log 8\pi \mp\frac {\pi }2),\\
}
where $[0;b_1,\cdots b_{\kappa}]$ is the continued fraction expansion of $a/q$. Then, 
\est{
\frac1q\sum_{a=1}^qf_{\pm}(a/q)^2=\frac1{\pi^2}(\log q)^4+c_{\pm}(\log q)^3+O(\log^2q),
}
for some real numbers $c_{\pm}$.
\end{corol}

We also remark that Corollary~\ref{Ypc} (and Lemma~\ref{lmfil} below) could also be used to give other results on continued fractions. For example, when combined with Burgess' bound  (or with estimates for Kloosterman sums, such as those in~\cite{DI}), Corollary~\ref{Ypc} gives non-trivial bounds for the average value of $f_{\pm}(a/q)$ as $a$ varies in short intervals.

Our approach in proving the reciprocity formula for $M(a,q)$ is different from that of Conrey and Young. Indeed our first step consists in relating $M\pr{a,q}$ to the Estermann function at the central point (for $q$ prime).  We remind that the Estermann function is defined as 
\est{
D\pr{s,\alpha,\tfrac aq}:=\sum_{n=1}^\infty\e{\tfrac {na}q}\frac{\sigma_\alpha(n)}{n^s},
} 
where $\e{z}:=e^{2\pi i z}$ and $\sigma_{\alpha}(n):=\sum_{d|n}d^\alpha$, for $\Re(s)>\max(1,1+\Re(\alpha))$ and extendable to an analytic function on $\C\setminus\{1,1+\alpha\}$.

The Estermann function is an extremely useful tool when studying moments of the Riemann zeta-function and of Dirichlet $L$-functions (see, for example,~\cite{BCH-B,Iwa,You11b}). This is mainly because it satisfies a functional equation, which is essentially equivalent to Voronoi's summation formula. Moreover, the values of the Estermann function at $s=0$ are related to important objects in number theory. In particular, one has (see, for example,~\cite{BC13b}) that 
\es{\label{ffsd}
\lim_{\alpha\rightarrow -1} \pr{D\pr{0,\alpha;\tfrac hk}+\tfrac12\zeta(-\alpha)}=\pi i \sss\pr{\tfrac hk},
}
where
\est{
\sss\pr{\frac hk}:=\sum_{m=1}^{k-1}\pr{\pr{\frac {mh}{k}}}\pr{\pr{\frac mk}}
}
is the Dedekind sum. (Here $((x)):=x-[x]-\frac12$ if $x\notin\Z$ and $((x))= 0$, if $x\in\Z$.) Moreover, 
\es{\label{ffcd}
D\pr{0,0;\tfrac hk}=\tfrac14+\tfrac i2\cc_0\pr{\tfrac hk}=\tfrac14-\tfrac i2 V\pr{\tfrac {\overline h}k},
}
where 
\est{
\cc_0\pr{\frac hk}:=-\sum_{m=1}^{k-1}\frac{m}{k}\cot\pr{\frac {\pi mh}{k}}
}
is the cotangent sum studied in~\cite{BC13a} and $V$ is the Vasyunin sum ($\overline h$ denotes the inverse of $h \mod k$), appearing in the Nyman-Beurling criterion for the Riemann hypothesis. 

Both $\cc_0$ and $\sss$ satisfy reciprocity relations, in Proposition~\ref{cfd} below we highlight yet another symmetry of $D$, obtained by using methods similar to those of~\cite{BC13a}. We will then use this result to deduce Theorems~\ref{mt} and~\ref{mtc} via the following relation between twisted moments of Dirichlet $L$-function and the Estermann function.

\begin{theo}\label{tff}
Let $q$ be a prime and let $s,z\in\C$. Let
\est{
M^*\pr{s,z;a,q}&:=\frac{q^{s-z}}{\varphi(q)}\sumstar_{\chi\mod q} L\pr{s-z,\overline\chi}L\pr{s+z,\chi}\chi(a)+\mbox{}\\
&\quad+\frac{q^{-z}}{\varphi(q)}\pr{q^{1-s}+q^s-q^{z}-q^{-z}}\zeta(s+z)\zeta(s-z).
}
Then,
\es{\label{tffe1}
M^*\pr{s,z;a,q}
&=\frac{\Gamma(1-s+z)}{(2\pi)^{1+z-s}}\bigg(e^{-\pi i\frac{1-s+z}2}D\pr{s+z,2s-1;\tfrac aq}+e^{\pi i\frac{1-s+z}2}D\pr{s+z,2s-1;-\tfrac aq}\bigg).\\
}
Equivalently,
\es{\label{tffe2}
D\pr{s+z,2s-1;a/q}=(2\pi)^{z-s}\Gamma(s-z)\pr{e^{\pi i\frac{s-z}2} M^*\pr{s,z;a,q}+e^{-\pi i\frac{s-z}2}M^*\pr{s,z;-a,q}}.
}
\end{theo}
\begin{remark}
The statements of Theorem~\ref{tff}, as well as several other formulae in the following sections, have to be interpreted as identities between meromorphic functions.
\end{remark}

In particular, we have the following relations between twisted second moments of $L(s,\chi)$ and special values of the Estermann function.
\begin{corol}\label{ctff}
Let $q$ be prime and let $(a,q)=1$. Then
\es{\label{ctffe1}
M_0^*\pr{a,q}
&=\frac12(1-i)D\pr{\frac12,0;\frac aq}+\frac12(1+i)D\pr{\frac12,0;-\frac aq}\bigg)
}
and
\begin{align}
&\frac{1}{\varphi(q)}\sumstar_{\chi\mod q} \pmd{L\pr{0,\chi}}^2\chi(a)=\frac1{\pi^2}\frac{q}{\varphi(q)}\sumstar_{\substack{\chi\mod q,\\ \chi(-1)=-1}} \pmd{L\pr{1,\chi}}^2\chi(a)=\sss\pr{\tfrac aq} \label{ctffe2},\\
&\frac{q}{\varphi(q)}\sumstar_{\substack{\chi\mod q,\\ \chi(-1)=-1}} L\pr{1,\overline\chi}L\pr{0,\chi}\chi(a)=\frac \pi2 c_0\pr{\tfrac aq}. \label{ctffe3}
\end{align}

\end{corol}
It is easy to generalize Theorem~\ref{tff} and its corollary to non-prime moduli. However, for simplicity we choose to deal with the prime case only, since Theorems~\ref{et},~\ref{mt} and~\ref{mtc} hold with such a neat formula only when $a$ and $q$ are prime.

We remark that an approximate version of~\eqref{ctffe1} in the special case $a=1$ appeared in~\cite{CG} and that~\eqref{ctffe3} has been recently proved by Louboutin~\cite{Lou} and Djankovic~\cite{Dja}, with a different method. In recent years, there has been quite a lot of interest on explicit formulae for second moments of Dirichlet $L$-functions (see, for example,~\cite{LZ,BR}) and Louboutin wonders weather one can obtain formulae similar to~\eqref{ctffe2} for the mean value of $L\pr{m,\overline\chi}L\pr{n,\chi}\chi(a)$, with $m,n\in\Z$ (the case when $m,n\geq1$ being studied in~\cite{Dja,BR}). Theorem~\ref{tff} can be used to obtain formulae of this kind, since $D(u,v,\frac aq)$ can be decomposed as a double sum of special values of the Hurwitz zeta-function and, if one among $u$ and $u-v$ is a non-positive integer, then one can execute one of the two sums (see for example the proof of Lemma~2 in~\cite{BC13b}).

Finally, we use a simple continued fraction argument, similar to that of Hickerson~\cite{Hic}, to prove a density result for
\est{
\eta\pr{\tfrac aq}:=\frac12(1-i)D\pr{\tfrac12,0;\tfrac aq}+\tfrac12(1+i)D\pr{\tfrac12,0;-\tfrac aq}.
}
\begin{theo}\label{cffd}
The set  $\{(\frac aq,\eta(\frac aq))\mid (a,q)\in\Z_{>0} \}$ is dense in $\R^2$. 
\end{theo}

It would be interesting to extend this result by adding the restriction that $q$ is prime, since in this case $\eta(\frac aq)$ coincides with $M_0(a,q)$. However, this restriction leads to a problem concerning the existence of primes in short intervals and arithmetic progressions, which doesn't seem to be easily tractable.

\section{Acknoledgments}
The author would like to thank Sary Drappeau, Adam Harper, James Maynard and Dimitris Koukoulopoulos for helpful discussions and Matthew Young and Andrew Granville for useful comments. 

\section{Proof of Theorem~\ref{tff}}
We start by proving Theorem~\ref{tff}. For $\Re(s)>1+|\Re(z)|$ and $(a,q)=1$, we define
\est{
A\pr{s,z;a,q}&:=\sum_{m \equiv -na\mod q}\frac{1}{(nm)^s}\pr{\frac{m}{n}}^{z}.\\
}
By dividing the sums into classes modulo $q$, one obtain that, for every $a,q$, $A\pr{s,z;a,q}$ extends to a meromorphic function on $\C^2$. In fact, one has
\es{\label{dfatvN}
A\pr{s,z;\frac aq}&=\frac1q\sum_{\ell=1}^qF\pr{s-z,\frac{\ell}q}F\pr{s+z,\frac{a\ell}q},\\
}
where $F(s,x)$ is the periodic zeta-function, defined as 
\est{
F(s,x):=\sum_{n=1}^\infty\frac{\e{nx}}{n^s}
}
for $\Re(s)>1$, $x\in\R$, and extendable to an analytic function of $s$ on $\C\setminus\{1\}$. (For this and for other properties of the periodic and the Hurwitz zeta-functions used below, see~\cite{Apo}, Chapter~11).

The following lemma relates $M^*\pr{s,z;a,q}$ to $A\pr{s,z;a,q}$. It is essentially equivalent to Lemma~2.4 of~\cite{You11a}, with the error term removed (in fact one can see that, by the functional equation of the Riemann zeta-function, the error terms discarded in the proof of Young's lemma combine and cancel exactly).

\begin{lemma}\label{ffs}
Let $q$ be prime or $q=1$ and let $(a,q)=1$. Let $s,z\in\C$. Then
\est{
M^*\pr{s,z;a,q}&=q^{s-z}A\pr{s,z;- a,q}.
}
\end{lemma}
\begin{proof}
We can assume $\Re(s)>1+|\Re(z)|$, since the lemma then follows without this restriction by analytic continuation. 

Expanding $L\pr{s-j,\overline\chi}$ and $L\pr{s+j,\chi}$ into their Dirichlet series, and applying the orthogonality relation for Dirichlet characters, we have that
\est{
\sumstar_{\chi\mod q} L\pr{s-z,\overline\chi}L\pr{s+z,\chi}\chi(a)&=\sumstar_{\chi\mod q} \sum_{m,n\geq1}\pr{\frac mn}^{z}\frac{\overline\chi(m)\chi(na)}{(mn)^s}\\
&=\varphi(q) \sum_{\substack{(nm,q)=1,\\m\equiv an\mod q}} \pr{\frac mn}^{z}\frac{1}{(mn)^s}-\sum_{\substack{(nm,q)=1}} \pr{\frac mn}^{z}\frac{1}{(mn)^s}.
}
Now, for $(a,q)=1$ we have that $m\equiv an\mod q$ and $q|mn$ imply $q|m,q|n$. It follows that
 \est{
\sum_{\substack{(q,nm)=1,\\m\equiv an\mod q}} \pr{\frac mn}^{z}\frac{1}{(mn)^s}&=A\pr{s,z;-\tfrac aq}-q^{-2s}\zeta(s+z)\zeta(s-z).
}
Moreover, we have
 \est{
\sum_{\substack{(q,nm)=1}} \pr{\frac mn}^{z}\frac{1}{(mn)^s}&=\pr{1-q^{z-s}-q^{-z-s}+q^{-2s}}\zeta(s+z)\zeta(s-z)
}
and the lemma follows.
\end{proof}

The following lemma (valid also for $q$ composite) expresses $A\pr{s,z;a,q}$ in terms of the Estermann function.
\begin{lemma}\label{ftrr}
Let $(a,q)=1$, $q>0$. Then
\es{\label{ftrq}
A\pr{s,z;a,q}
&=q^{-s+z}\Gamma(1-s+z)(2\pi)^{s-1-z}\times\\
&\quad\times\bigg(e^{\pi i\frac{1-s+z}2}D\pr{s+z,2s-1;\tfrac aq}+e^{-\pi i\frac{1-s+z}2}D\pr{s+z,2s-1;-\tfrac aq}\bigg).\\
}
\end{lemma}
\begin{proof}
We start by the decomposition~\eqref{dfatvN} of $A\pr{s,z;a,q}$ in terms of the periodic zeta-function and we decompose further one of the two periodic zeta-function using the identity
\es{\label{ffadsa}
F\pr{s,\tfrac hk}=k^{-s}\sum_{r=1}^k\e{\tfrac{rh}k}\zeta\pr{s,\tfrac rk}.
}
As usual, $\zeta\pr{s,x}$ is the the Hurwitz zeta-function, defined as
\est{
\zeta\pr{s,x}:=\sum_{n\geq\pg{1-x}}\frac{1}{(n+{x})^s}
}
for $\Re(s)>1$, and extendable to a meromorphic function of $\C$. By~\eqref{dfatvN} and~\eqref{ffadsa} we get
\est{
A\pr{s,z;a,q}&=q^{-1-s-z}\sum_{r,\ell=1}^q\e{r\ell\tfrac{ a}{q}}F(s-z,\tfrac{\ell}q)\zeta(s+z,\tfrac{r}q).\\
}
Next, we express also the other periodic zeta-function in terms of the Hurwitz zeta-function, but this time we use the functional equation,
\est{
F(s-z,x)=\Gamma(1-s+z)(2\pi)^{s-1-z}\pr{e^{\pi i\frac{1-s+z}2}\zeta(1-s+z,x)+e^{-\pi i\frac{1-s+z}2}\zeta(1-s+z,-x)}.
}
We get
\es{\label{ll1}
A\pr{s,z;a,q}&=q^{-1-s-z}\Gamma(1-s+z)(2\pi)^{s-1-z}\sum_{r,\ell=1}^q\e{r\ell\tfrac aq}\times\\
&\quad\times\Big(e^{\pi i\frac{1-s+z}2}\zeta(1-s+z,\tfrac\ell q)\zeta(s+z,\tfrac {r}q )+e^{-\pi i\frac{1-s+z}2}\zeta(1-s+z,-\tfrac\ell q)\zeta(s+z,\tfrac {r}q )\Big).\\
}
Now, for $1-\Re(z)<\Re(s)<\Re(z)$ we have
\es{\label{ll2}
 \sum_{r,\ell=1}^q\e{r\ell\tfrac aq}\zeta(1-s+z,\pm\tfrac\ell q)\zeta(s+z,\tfrac {r}q)&= \sum_{n\geq\pg{1\mp\frac\ell q}}\sum_{m\geq\pg{1-\frac rq}}\sum_{r,\ell=1}^q\frac{\e{r\ell\frac aq}}{\pr{n\pm\frac\ell q}^{1-s+z}\pr{n+\frac r q}^{s+z}}\\
 &=q^{1+2z} \sum_{n,m\geq1}\frac{\e{\pm nm\frac aq}}{n^{1-s+z}m^{s+z}}\\
 &=q^{1+2z} D(s+z,2s-1,\pm \tfrac aq).
}
Thus, combining~\eqref{ll1} and~\eqref{ll2}, we obtain~\eqref{ftrq} for $1-\Re(z)<\Re(s)<\Re(z)$ and the Lemma then follows by analytic continuation.
\end{proof}

\begin{proof}[Proof of Theorem~\ref{tff}]
Equation~\eqref{tffe1} follows immediately by Lemma~\ref{ffs} and Lemma~\ref{ftrr}. Equation~\eqref{tffe2} follows from~\eqref{tffe1} by using the reflection formula for the Gamma function,
\es{\label{rfgf}
\Gamma(s)\Gamma(1-s)=\frac{\pi}{\sin (\pi s)}.
}
\end{proof}

\begin{proof}[Proof of Corollary~\ref{ctff}]
Equation~\eqref{ctffe2} is simply~\eqref{tffe1} with $s=\frac12$, $w=0$. We now prove~\eqref{ctffe2}. By~\eqref{tffe1}, with $s=-z$, and~\eqref{ffsd}, we have 
\est{
&M^*\pr{-z,z;a,q}
=\frac{\Gamma(1+2z)}{(2\pi)^{1+2z}}\bigg(e^{-\pi i\frac{1+2z}2}D\pr{0,-2z-1;\tfrac aq}+e^{\pi i\frac{1+2z}2}D\pr{0,-2z-1;-\tfrac aq}\bigg)\\
&\quad=\frac{\Gamma(1+2z)}{(2\pi)^{1+2z}}\bigg(e^{-\pi i\frac{1+2z}2}(\pi i \sss\pr{\tfrac aq}-\tfrac12\zeta(1+2z))+e^{\pi i\frac{1+2z}2}(-\pi i \sss\pr{\tfrac aq}-\tfrac12\zeta(1+2z))\bigg)+o(1)\\
&\quad=\frac{\Gamma(1+2z)}{(2\pi)^{1+2z}}\Big(2\pi \sss\pr{\tfrac aq} \sin\pr{ \tfrac{\pi}2(1+2z)}-\zeta(1+2z)\cos\pr{ \tfrac{\pi}2(1+2z)}\Big)+o(1)\\
}
as $z\rightarrow 0$, since, $\sss(-\frac aq)=-\sss(\frac aq)$. Thus, taking the limit for $z\rightarrow0$, we get
 \es{\label{asd11}
M^*\pr{0,0;a,q}
&=s\pr{\tfrac aq} +\frac 14.\\
}
Finally, by definition,
\es{\label{asd21}
M^*\pr{0,0;a,q}
&=\frac{1}{\varphi(q)}\sumstar_{\chi\mod q} \pmd{L\pr{0,\chi}}^2\chi(a)+\frac14\\
&=\frac1{\pi^2}\frac{q}{\varphi(q)}\sumstar_{\substack{\chi\mod q,\\\chi(-1)=-1}} \pmd{L\pr{1,\chi}}^2\chi(a)+\frac14,\\
}
since $\zeta(0)=-\frac12$ and $|L(0,\chi)|=\frac{q^\frac12}{\pi}|L(1,\chi)|$ by the functional equation. Equation~\eqref{ctffe2} then follows by~\eqref{asd11} and~\eqref{asd21}.

Equation~\eqref{ctffe3} follows in a similar way, using~\eqref{tffe2} and~\eqref{ffcd}.
\end{proof}

\section{Proofs of Theorems~\ref{et},~\ref{mt} and~\ref{mtc}}
In this section we prove the reciprocity formula for the Estermann function (valid also for composite $a,q$). Theorem~\ref{et} and Theorem~\ref{mtc} (as well as its weaker version, Theorem~\ref{mt}) then follow immediately by combining it with Theorem~\ref{tff}. The strategy of the proof is similar to the one used in~\cite{BC13a} for the exact formula and Lemma A of~\cite{KP02} for the approximated version (however, in our case the comparison with Gauss' hypergeometric formula simplify considerably the argument).

For convenience of notation, we write
\est{
D_j\pr{s; \frac aq}:=D\pr{s+j,2s-1; \frac aq}.
}
We also define a twisted Eisenstein series
\est{
\s_j\pr{s,z}:=\frac1{(2\pi i)^j}\sum_{n\geq1}\e{nz}\frac{\sigma_{2s}(n)}{n^{s+\frac12+j}},
}
defined for $\Im(z)\geq0$ if $|\Re(s)|<\frac12+j-1$ and $\Im(z)>0$ otherwise. Notice that the $\s_j\pr{s,z}$ can be interpreted as Eichler integrals of $\s_0\pr{s,z}.$

\begin{lemma}\label{rtfs}
Let $|\Re(s)|<\frac12$, $s\neq0$. Then for $0<\arg z<\pi$ we have 
\es{\label{mmmf}
&\s_0\pr{s,z}=\\
&=\sum_{j=0}^\infty \frac{(-1)^jQ_{2j}(s)}{j!}\bigg(z^j\s_j\pr{s,-1/z}-\sum_{\ell=0}^j\frac{(-1)^{j-\ell}}{(j-\ell)!}\zeta\pr{\tfrac12+\ell+s}\zeta\pr{\tfrac12+\ell-s}\pr{\frac z{2\pi i}}^{\ell}+\\
&\hspace{1cm}-\Gamma\pr{\tfrac12+s-j}\zeta\pr{1+2s}(2\pi i/ z)^{-\frac12-s}-\Gamma\pr{\tfrac12-s-j}\zeta\pr{1-2s}(2\pi i/ z)^{-\frac12+s}\bigg)+Z(s,z),
}
where
\est{
Z(s,z)&:=\frac1{2\pi i}\int_{\pr{\frac12+\delta}}\Gamma(w) \frac{\cos\pi s}{\sin\pi w}\zeta\pr{\tfrac12+w-s}\zeta\pr{\tfrac12+w+s}(-2\pi iz)^{-w}dw,\\
Q_{2j}(s)&:=\prod_{\ell=0}^{j-1}\pr{(\tfrac12+\ell)^2-s^2}
}
for any $|\Re(s)|<\delta<\frac12$. Moreover, for any $N\geq1$ we have
\es{\label{mmmfao}
&\s_0\pr{s,z}=\\
&=\sum_{j=0}^{2N} \frac{(-1)^jQ_{2j}(s)}{j!}\bigg(z^j\s_j\pr{s,-1/z}+\frac{\zeta\pr{\tfrac12+j+s}\zeta\pr{\tfrac12+j-s}(z/2\pi i)^{j}}{j!}\bigg)+{}\\
&\quad+\Gamma\pr{\tfrac12+s}\zeta\pr{1+2s}(-2\pi iz)^{-\frac12-s}+\Gamma\pr{\tfrac12-s}\zeta\pr{1-2s}(-2\pi iz)^{-\frac12+s}+\mathcal{E}_N(s,z),
}
where $\mathcal{E}_N(s,z)$ is analytic as a function of $z$ in $\Im(z)>0$ and $\mathcal C^N$ in $\Im(z)\geq0$. Also, $\mathcal{E}_N(s,z)$ satisfies 
\begin{align}
\mathcal{E}_N(s,z)&\ll_{N,\eps}(1+|s|^{4N+4})\pr{|z|^{N+\frac32-\eps}+|z|^{2N+1}},\label{ffb1}\\
\mathcal{E}_N(s,z)&\ll_{N,\eps}(1+|s|^{8N+4})|z|^{2N+1},\label{ffb2}
\end{align}
where $\eps>0$ and $\Im(z)\geq0$.

\end{lemma}
\begin{remark}\label{rkexp}
In the Lemma,  $(i/ z)^{s}$ and $(-iz)^{s}$ have to be interpreted as $e^{s(\frac {\pi i}2-\log z)}$ and $e^{s(\log z-\frac{\pi i} 2)}$ respectively, with $0\leq\arg z\leq\pi$.
\end{remark}
\begin{remark}\label{rk2}
Notice that if we denote by $\tilde {\mathcal{E}}_N$ the sum of $\mathcal{E}_N$ and the terms from $N<j\leq 2N$ of the sum on the right hand side of~\eqref{mmmfao}, then $\tilde{\mathcal{E}}_N(s,z)$ is still bounded by $(1+|s|^{4N+4})|z|^{N+1}$, and is $C^N$ in $\R\setminus\{0\}$ and $C^{[N/2]}$ at $z=0$.

We also remark that we chose to give bounds for $\mathcal{E}_N(s,z)$ which are uniform in $s$, since this is needed if one wants to use Lemma~\ref{rtfs} to give results analogous to the Theorems~\ref{mt} and~\ref{mtc} for the function $S_f$ studied by Young in~\cite{You11a}.

\end{remark}
\begin{proof}
For $\Re(w)>\frac12+|\sigma|$ we have the Dirichlet expansion
\est{
\zeta\pr{\tfrac12+w+s}\zeta\pr{\tfrac12+w-s}=\sum_{n=1}^\infty\frac{\sigma_{2s}(n)}{n^{\frac12+w+s}},
}
and thus, expressing the exponential as a Mellin transform, we can write $\s\pr{s,z}$ as 
\es{\label{sfh}
\s_0\pr{s,z}&=\frac1{2\pi i}\int_{\pr{\tfrac12+\delta}}\Gamma(w)\zeta\pr{\tfrac12+w+s}\zeta\pr{\tfrac12+w-s}(-2\pi iz)^{-w}dw.\\
}
Now, we apply the functional equation $\zeta(1-s)=\zeta(s)/\chi(s)$, where
\est{
\chi(1-s):=2(2\pi)^{-s}\Gamma\pr{s}\cos\frac{\pi s}2,
}
getting
\est{
\s_0\pr{s,z}&=\frac1{2\pi i}\int_{\pr{-\frac12-\delta}}H(w,s) \zeta\pr{\tfrac12+w-s}\zeta\pr{\tfrac12+w+s}(2\pi i/z)^{-w}dw,\\
}
after the change of variable $w\rightarrow -w$ (notice that in this context we have the equality $\arg(i/z)=-\arg(-iz)$), where 
\es{\label{ffh}
H(w,s)
&:=\frac{(2\pi)^{2w}\Gamma(-w)}{\chi\pr{\tfrac12+w-s}\chi\pr{\tfrac12+w+s}}\\
&=\frac\pi2\frac{\Gamma(-w)}{\Gamma\pr{\frac12-w-s}\Gamma\pr{\frac12-w+s}}\frac1{\cos\pr{\frac\pi2\pr{\frac12-w-s}}\cos\pr{\frac\pi2\pr{\frac12-w+s}}}\\
&=\pi\frac{\Gamma(-w)}{\Gamma\pr{\frac12-w-s}\Gamma\pr{\frac12-w+s}\sin\pi w}\pr{1-\frac{\cos\pi s}{\sin\pi w+\cos \pi s}}.\\
&=\frac{\Gamma(-w)\Gamma(1-w)\Gamma(w)}{\Gamma\pr{\frac12-w-s}\Gamma\pr{\frac12-w+s}}\pr{1-\frac{\cos\pi s}{\sin\pi w+\cos \pi s}},\\
}
by the reflection formula for the Gamma function~\eqref{rfgf}.
Since on the line of integration we have $\Re(w)<0$, we can use Gauss' hypergeometric formula~\eqref{ghf}, getting
\es{\label{ffh2}
\frac{\Gamma(-w)\Gamma(1-w)\Gamma(w)}{\Gamma\pr{\frac12-w-s}\Gamma\pr{\frac12-w+s}}
&={}_2F_1 \pr{\tfrac12+s,\tfrac12-s;1-w;1}\,\Gamma(w)\\
&=\sum_{j=0}^\infty \frac{Q_{2j}(s)}{j!(1-w)(2-w)\cdots(j-w)}\Gamma(w)\\
&=\sum_{j=0}^\infty (-1)^j\frac{Q_{2j}(s)\Gamma(w-j)}{j!}.
}
Moreover, going backwards,
\est{
\pi\frac{\Gamma(-w)}{\Gamma\pr{\frac12-w-s}\Gamma\pr{\frac12-w+s}\sin\pi w}\frac{\cos\pi s}{\sin\pi w+\cos \pi s}=\frac{(2\pi)^{2w}\Gamma(-w)}{\chi\pr{\tfrac12+w-s}\chi\pr{\tfrac12+w+s}}\frac{\cos\pi s}{\sin\pi w}.
}
Thus,
\est{
H(w,s)=\sum_{j=0}^\infty (-1)^j\frac{Q_{2j}(s)\Gamma(w-j)}{j!}-\frac{(2\pi)^{2w}\Gamma(-w)}{\chi\pr{\tfrac12+w-s}\chi\pr{\tfrac12+w+s}}\frac{\cos\pi s}{\sin\pi w}
}
and so
\est{
\s_0\pr{s,z}=\sum_{j=0}^\infty (-1)^j\frac{Q_{2j}(s)}{j!}\psi_j(s,z)+Z(s,z),
}
where
\est{
\psi_j(s):=\frac1{2\pi i}\int_{\pr{-\frac12-\eps}}\Gamma(w-j) \zeta\pr{\tfrac12+w-s}\zeta\pr{\tfrac12+w+s}(2\pi i/z)^{-w}dw
}
and the change of the order of summation and integration is justified by absolute convergence. Now, by contour integration,
\est{
\psi_j(s)&=\frac1{2\pi i}\int_{\pr{-\frac12-\eps}}\Gamma(w-j) \zeta\pr{\tfrac12+w-s}\zeta\pr{\tfrac12+w+s}(2\pi i/z)^{-w}dw\\
&=\frac1{2\pi i}\int_{\pr{\frac12+j}}\Gamma(w-j)\zeta\pr{\tfrac12+w+s}\zeta\pr{\tfrac12+w-s}(2\pi i/ z)^{-w}dw+{}\\
&\quad-\sum_{\ell=0}^j\frac{(-1)^{j-\ell}}{(j-\ell)!}\zeta\pr{\tfrac12+\ell+s}\zeta\pr{\tfrac12+\ell-s}(2\pi i/ z)^{-\ell}+{}\\
&\quad-\Gamma\pr{\tfrac12+s-j}\zeta\pr{1+2s}(2\pi i/ z)^{-\frac12-s}-\Gamma\pr{\tfrac12-s-j}\zeta\pr{1-2s}(2\pi i/ z)^{-\frac12+s}
}
and~\eqref{mmmf} follows since
\est{
\frac1{2\pi i}\int_{\pr{\frac12+j}}\Gamma(w-j)\zeta\pr{\tfrac12+w+s}\zeta\pr{\tfrac12+w-s}(2\pi i/ z)^{-w}dw
&=z^j\s_j\pr{s,-1/z}.
}
To prove~\eqref{mmmfao}, we can proceed in the following way. We start again form~\eqref{sfh}, but this time we move the line of integration to $\Re(w)=-N-\frac32+\eps$, with $0<\eps<\frac12$, $N\geq0$. We get
\es{\label{ann}
\s\pr{s,z}&=\frac1{2\pi i}\int_{\pr{-N-\frac32+\eps}}\Gamma(w)\zeta\pr{\tfrac12+w+s}\zeta\pr{\tfrac12+w-s}(-2\pi iz)^{-w}dw+{}\\
&\quad+\sum_{j=0}^{N+1}\frac{(-1)^j\zeta\pr{\tfrac12+s-j}\zeta\pr{\tfrac12-s-j}(-2\pi iz)^{j}}{j!}+{}\\
&\quad+\Gamma\pr{\tfrac12+s}\zeta\pr{1+2s}(-2\pi iz)^{-\frac12-s}+\Gamma\pr{\tfrac12-s}\zeta\pr{1-2s}(-2\pi iz)^{-\frac12+s}.
}
Notice that
\est{
\chi\pr{\tfrac12+s-j}\chi\pr{\tfrac12-s-j}
&=2(2\pi)^{-1-2j}\cos\pr{\pi s}\Gamma\pr{\tfrac12+j-s}\Gamma\pr{\tfrac12+j+s}\\
&=(2\pi)^{-2j}Q_{2j}(s)\\
}
by the reflection formula for the Gamma function (and $s\Gamma(s)=\Gamma(s+1)$) and thus the sum on the second line of~\eqref{ann} is equal to
\est{
\sum_{j=0}^{N+1}(-1)^jQ_{2j}(s)\zeta\pr{\tfrac12+s+j}\zeta\pr{\tfrac12-s+j}\frac{(z/2\pi i)^{j}}{j!}.
}
Moreover, using the functional equation and making the change of variable $w\rightarrow-w$, we get that the integral on the right hand side of~\eqref{ann} is equal to
\es{\label{fsfd}
&\frac1{2\pi i}\int_{\pr{N+\frac32-\eps}}H(w,s)\zeta\pr{\tfrac12+w+s}\zeta\pr{\tfrac12+w-s}(2\pi i/z)^{-w}dw.
}
Now, from~\eqref{ffh} we have that
\est{
H(w,s)=\frac{\Gamma(-w)\Gamma(1-w)\Gamma(w)}{\Gamma\pr{\frac12-w-s}\Gamma\pr{\frac12-w+s}}\pr{1+O\pr{e^{-\pi(|w|-|s|)}}}\\
}
and, if $|s|^2=o(|w|)$, $\Re(w),\Re(s)\ll1$ (with $|w-m|\gg1$ for all $m\in\Z$), then by Stirling's formula (see for example~\cite{KP11}) for any $N\geq1$ we have 
\est{
\frac{\Gamma(-w)\Gamma(1-w)}{\Gamma\pr{\frac12-w-s}\Gamma\pr{\frac12-w+s}}&=\exp\bigg(\sum_{j=0}^N \frac{R_{1,j+1}(s)}{w^j}+O_N\pr{\frac{1+|s|^{N+2}}{|w|^{N+1}}}\bigg)\\
&=\sum_{j=0}^N \frac{R_{2,2j}(s)}{w^j}+O_N\pr{\frac{1+|s|^{2N+2}}{|w|^{N+1}}}\\
&=\sum_{j=0}^N \frac{R_{3,2j}(s)}{j!(1-w)\cdots(j-w)}+O_N\pr{\frac{1+|s|^{2N+2}}{|w|^{N+1}}}
}
as $\Im(w)\rightarrow\infty$, for certain polynomials $R_{i,j}$, of degree $j$. Thus, for $|s|^2=o(|w|)$ (and $\Re(w),\Re(s)\ll1$) we have 
\es{\label{gggg}
H(w,s)=\sum_{j=0}^N \frac{Q_{2j}(s)\Gamma(w-j)}{j!}+O_N\pr{\frac{1+|s|^{2N+2}|\Gamma(w)|}{|w|^{N+1}}},
}
since by~\eqref{ffh2} we must have $R_{3,2j}(s)=Q_{2j}(s)$. Moreover, a trivial bound for $H(w,s)$, using Stirling's formula, shows that~\eqref{gggg} holds also when $|w|=O(|s|^{2})$.

Thus, applying~\eqref{gggg} with $N=2N+1$, we get that~\eqref{fsfd} is equal to
\est{
&\frac1{2\pi i}\int_{\pr{N+\frac32-\eps}}H(w,s)\zeta\pr{\tfrac12+w+s}\zeta\pr{\tfrac12+w-s}(2\pi i/z)^{-w}dw=\sum_{j=0}^{2N+1} \frac{Q_{2j}(s)\Theta_j(z,s)}{j!}+\mathcal{R}_{2N+1}(z,s),
}
where
\es{\label{erty}
\Theta_j(z,s)&:=\frac1{2\pi i}\int_{\pr{N+\frac32-\eps}}\Gamma(w-j)\zeta\pr{\tfrac12+w+s}\zeta\pr{\tfrac12+w-s}(2\pi i/z)^{-w}dw\\
&=z^j\s_j\pr{s,-1/z}+\sum_{N+2\leq r\leq j}\frac{(-1)^{j-r}\zeta\pr{\tfrac12+r+s}\zeta\pr{\tfrac12+r-s}(-2\pi iz)^{r}}{(j-r)!}.\\
}
The function $\mathcal{R}_{2N+1}(z,s)$ is holomorphic in $\Im(z)>0$ and $\mathcal C^{N}$ in $\R$, since for $\ell\leq N$ and $0\leq\arg z\leq\pi$ we have
\est{
\mathcal{E}_{2N+1}^{(\ell)} (z,s)&\ll |z|^{N-\ell+\frac32-\eps}\int_{\pr{N+1+\eps}}|\zeta\pr{\tfrac12+w+s}\zeta\pr{\tfrac12+w-s}|\frac{1+|s|^{4N+4}}{|w|^{N-\ell+\eps}}|dw|\\
&\ll \pr{1+|s|^{4N+4}}|z|^{N-\ell+\frac32-\eps}.
}
Thus, we have
\est{
\s\pr{s,z}&=\sum_{j=0}^{2N+1} \frac{Q_{2j}(s)}{j!}z^j\s_j\pr{s,-1/z}+\sum_{j=0}^{2N+1}(-1)^{j} \frac{Q_{2j}(s)}{j!}\sum_{r=N+2}^j\frac{\zeta\pr{\tfrac12+r+s}\zeta\pr{\tfrac12+r-s}(2\pi iz)^{r}}{(j-r)!}+{}\\
&\quad+\mathcal{R}_{2N+1}(z,s)+\sum_{j=0}^{N+1}(-1)^jQ_{2j}(s)\zeta\pr{\tfrac12+s+j}\zeta\pr{\tfrac12-s+j}\frac{(z/2\pi i)^{j}}{j!}+{}\\
&\quad+\Gamma\pr{\tfrac12+s}\zeta\pr{1+2s}(-2\pi iz)^{-\frac12-s}+\Gamma\pr{\tfrac12-s}\zeta\pr{1-2s}(-2\pi iz)^{-\frac12+s}.
}
Thus,~\eqref{mmmfao} and the bound~\eqref{ffb1} follow by writing
\est{
\mathcal{E}_{N}(s,z)&:=\frac{Q_{2(N+1)}(s)}{(N+1)!}z^{2N+1}\s_{2N+1}\pr{s,-1/z}-\sum_{j=N+2}^{2N}(-1)^jQ_{2j}(s)\zeta\pr{\tfrac12+s+j}\zeta\pr{\tfrac12-s+j}\frac{(z/2\pi i)^{j}}{j!}+{}\\
&\quad+\sum_{j=0}^{2N+1} \frac{Q_{2j}(s)}{j!}\sum_{N+2\leq r\leq j}\frac{(-1)^{j-r}\zeta\pr{\tfrac12+r+s}\zeta\pr{\tfrac12+r-s}(-2\pi iz)^{r}}{(j-r)!}+\mathcal{R}_{2N+1}(z,s).
}
(If $N=0$ the second addend on the first line has to be replaced by $-Q_{2}(s)\zeta\pr{\tfrac32+s}\zeta\pr{\tfrac32-s}\frac {z}{2\pi i}$). The bound~\eqref{ffb2} follows by applying~\eqref{mmmfao} with $2N$ in place of $N$ and bounding trivially the extra terms in the sum.
\end{proof}

\begin{prop}\label{cfd}

Let $a,q\in\Z_{>0}$, with $(a,q)=1$ and let $|\Re(s)|<\frac12$, $s\neq0$. Then
\es{\label{ecfd}
&D_0(s+\tfrac12,\pm\tfrac aq)=\\
&=\sum_{j=0}^\infty \frac{(-1)^jQ_{2j}(s)}{j!}\bigg(\pr{\frac{\pm a}{2\pi iq}}^jD_j(s+\tfrac12,\mp\tfrac qa)-\sum_{\ell=0}^j\frac{(-1)^{j-\ell}}{(j-\ell)!}\zeta\pr{\tfrac12+\ell+s}\zeta\pr{\tfrac12+\ell-s}\pr{\frac {\pm a}{2\pi iq}}^{\ell}+\\
&\hspace{1cm}-\Gamma\pr{\tfrac12+s-j}\zeta\pr{1+2s}( 2\pi q/a)^{-(\frac12+s)}e^{\mp(\frac12+s)\frac {\pi i}2}+\\
&\hspace{1cm}-\Gamma\pr{\tfrac12-s-j}\zeta\pr{1-2s}( 2\pi q/a)^{-(\frac12-s)}e^{\mp(\frac12-s)\frac {\pi i}2}\bigg)+Z_\pm\pr{s,\frac aq},
}
where for  $|\Re(s)|<\delta<\frac12$
\est{
Z_\pm(s,z)&:=\frac1{2\pi i}\int_{\pr{\frac12+\delta}}\Gamma(w) \frac{\cos\pi s}{\sin\pi w}\zeta\pr{\tfrac12+w-s}\zeta\pr{\tfrac12+w+s}(2\pi z)^{-w}e^{\pm\frac {\pi i w}2}dw.\\
}
Moreover, for $N\geq0$
\es{\label{mmmfa}
&D_0(s+\tfrac12,\pm \tfrac aq)-\mathcal E_N\pr{s,\pm \tfrac aq}=\\
&=\sum_{j=0}^{2N} \frac{(-1)^jQ_{2j}(s)}{j!}\bigg(\pr{\frac{\pm a}{2\pi iq}}^jD_j(s+\tfrac12,\mp \tfrac qa)+\frac{\zeta\pr{\tfrac12+j+s}\zeta\pr{\tfrac12+j-s}}{j!}\pr{\frac{\pm a}{2\pi iq}}^j\bigg)+{}\\
&\quad+\Gamma\pr{\tfrac12+s}\zeta\pr{1+2s}(2\pi a/q)^{-(\frac12+s)}e^{\pm(\frac12+s)\frac {\pi i}2} +\Gamma\pr{\tfrac12-s}\zeta\pr{1-2s}(2\pi a/q)^{-(\frac12-s)}e^{\pm(\frac12-s)\frac {\pi i}2},
}
where $\mathcal E_N\pr{s,x}$ is as in Lemma~\ref{rtfs}.

\end{prop}
\begin{proof}
Let $\eps>0$ and let $z=\pm \frac aq\pr{1\pm i\eps}$. Then, we have
\est{
\s\pr{s,z}&=\sum_{n\geq1}\e{ in\tfrac aq \eps}\e{\pm n\tfrac aq}\frac{\sigma_{2s}(n)}{n^{s+\frac12}}\\
&=\frac1{2\pi i}\int_{(2)}\Gamma(w)D\pr{s+w+\tfrac12,2s,\pm \tfrac aq}(2\pi \tfrac aq\eps)^{-w}dw.\\
}
Now, $D\pr{s+w+\frac12,2s,\pm \frac aq}$ has singularities at $w=\frac12\mp s$, with residues $q^{-1+2s}\zeta(1-2s)$ and $q^{-1-2s}\zeta(1+2s)$ respectively. Thus, moving the line to $\Re(w)=-\xi$ for some $0<\xi<1$, we get
\es{\label{inz}
\s\pr{s,\pm \tfrac aq\pr{1\pm i\eps}}&=D_0\pr{s+\tfrac12,\pm \tfrac aq}+\Gamma\pr{\tfrac12-s}\zeta(1-2s)( 2\pi  aq\eps)^{-\frac12+s}+\\
&\quad+\Gamma\pr{\tfrac12+s}\zeta(1+2s)(2\pi  aq\eps)^{-\frac12-s}+O_{s,\frac aq}(\eps^\xi).\\
}
In the same way, we can write $-\frac1z=\mp \frac qa\pr{1\mp i\eps'}$, $\eps'=\frac{ \eps}{1\pm i\eps}=\eps+O(\eps^2)$ and thus 
\es{\label{in1z}
\s\pr{s,-\tfrac 1z}&=D_0\pr{s+\tfrac12,\mp \tfrac qa}+\Gamma\pr{\tfrac12-s}\zeta(1-2s)(2\pi  qa\eps)^{-\frac12+s}\pr{1+O_s(\eps)}+{}\\
&\quad+\zeta(1+2s)\Gamma\pr{\tfrac12+s}(2\pi  qa\eps)^{-\frac12-s}(1+O_s(\eps))+O_{s,\frac aq}(\eps^\xi)\\
&=D_0\pr{s+\tfrac12,-\tfrac qa}+\Gamma\pr{\tfrac12-s}\zeta(1-2s)(2\pi  qa\eps)^{-\frac12+s}+{}\\
&\quad+\zeta(1+2s)\Gamma\pr{\tfrac12+s}(2\pi  qa\eps)^{-\frac12-s}+O_{s,\frac aq}(\eps^\xi+\eps^{\frac12-|\sigma|}).\\
}
Therefore, from~\eqref{inz} and~\eqref{in1z} it follows that
\est{
\s\pr{s,z}-\s\pr{s,-\tfrac 1z}&=D\big(s+\tfrac12,2s,\pm \tfrac aq\big)-D\big(s+\tfrac12,2s,\mp \tfrac qa\big)+O_{s,\tfrac aq}(\eps^\varepsilon).\\
}
Moreover, by Remark~\ref{rkexp} we have
\est{
\lim_{\eps\rightarrow 0}(2\pi i/ z)^{-\frac12-s}&=e^{-(\frac12+s)\log 2\pi \frac q{a}-(\frac12+s)(\frac{\pi i} 2+\pi i(\pm1-1)/2}\\
&=( 2\pi q/a)^{-(\frac12+s)}e^{\mp(\frac12+s)\frac {\pi i}2}\\
\lim_{\substack{\eps\rightarrow 0}}(-2\pi i z)^{-\frac12-s}&=e^{-(\frac12+s)\log 2\pi \frac {a}{q}-(\frac12+s)(-\frac{\pi i} 2-\pi i(\sgn a-1)/2}\\
&=(2\pi a/q)^{-(\frac12+s)}e^{\pm (\frac12+s)\frac {\pi i}2}\\
}
and similarly $\lim_{\eps\rightarrow 0} Z(s,\pm a/q)=Z_{\pm}(s,a/q)$. Equation~\eqref{ecfd} then follows by letting $\eps$ go to zero in~\eqref{mmmf}, upon noticing that for $j\geq1$ and $|\Re(s)|<\frac12$ we have $\s_j\pr{s,\mp\frac qa}=\frac1{(2\pi i)^j}D_j\pr{s+\frac12,\mp\frac qa}$.

Equation~\eqref{mmmfa} follows in the same way.
\end{proof}

Taking $s\rightarrow 0$, Proposition~\ref{cfd} gives the following corollary. 

\begin{corol}\label{mcfd}
 Let $a,q\in\Z_{>0}$, with $(a,q)=1$. Let $|\Re(s)|<\frac12$. Then
\es{\label{ftr1}
D_0(\tfrac12,\pm\tfrac aq)
&=\frac1{\pi}\sum_{j=0}^\infty \frac{(-1)^j\Gamma(\tfrac12+j)^2}{j!}\bigg(\pr{\frac{\pm a}{2\pi iq}}^jD_j(\tfrac12,\mp\tfrac qa)-\sum_{\ell=0}^j\frac{(-1)^{j-\ell}}{(j-\ell)!}\zeta\pr{\tfrac12+\ell}^2\pr{\frac {\pm a}{2\pi iq}}^{\ell}+\\
&\quad-\frac12\Gamma(\tfrac12-j)( \pi q/a)^{-\frac12}(1\mp i)(\Psi(\tfrac12-j)+2\gamma-\log( 2\pi q/a)\mp\frac{\pi i}{2})\bigg)+Z_\pm\pr{0,\frac aq}
}
and
\es{\label{ftr}
&D_0(\tfrac12,\pm \tfrac aq)-\mathcal E_N\pr{0,\pm \tfrac aq}=\\
&=\frac1{\pi}\sum_{j=0}^{2N} \frac{(-1)^j\Gamma(\tfrac12+j)^2}{j!}\bigg(\pr{\frac{\pm a}{2\pi iq}}^jD_j(\tfrac12,\mp \tfrac qa)+\frac{\zeta\pr{\tfrac12+j}^2}{j!}\pr{\frac{\pm a}{2\pi iq}}^j\bigg)+{}\\
&\quad+\frac12\pr{\frac {q}{a}}^{\frac12}  (\log (q/a)+\gamma -\log 8\pi -\frac {\pi }2)\pm \frac i2\pr{\frac {q}{a}}^{\frac12}  (\log (q/a)+\gamma -\log 8\pi +\frac {\pi }2).
}
\end{corol}
\begin{proof}
As $s\rightarrow0$, we have
\est{
&\Gamma\pr{\tfrac12+s-j}\zeta\pr{1+2s}( 2\pi q/a)^{-(\frac12+s)}e^{\mp(\frac12+s)\frac {\pi i}2}+\Gamma\pr{\tfrac12-s-j}\zeta\pr{1-2s}( 2\pi q/a)^{-(\frac12-s)}e^{\mp(\frac12-s)\frac {\pi i}2}\\
&=\Gamma(\tfrac12-j)(1+\Psi(\tfrac12-j)s)(\tfrac{1}{2s}+\gamma )( 2\pi q/a)^{-\frac12}(1-\log( 2\pi q/a) s)e^{\mp \frac {\pi i}4}(1\mp\frac{\pi i}{2}s)+{}\\
&\quad+\Gamma(\tfrac12-j)(1-\Psi(\tfrac12-j)s)(-\tfrac{1}{2s}+\gamma )( 2\pi q/a)^{-\frac12}(1+\log( 2\pi q/a) s)e^{\mp \frac {\pi i}4}(1\pm\frac{\pi i}{2}s)+O(s)\\
&=\frac12\Gamma(\tfrac12-j)( \pi q/a)^{-\frac12}(1\mp i)(\Psi(\tfrac12-j)+2\gamma-\log( 2\pi q/a)\mp\frac{\pi i}{2})+O(s)\\
}
Thus,~\eqref{ftr} follows by taking $s\rightarrow 0$ in~\eqref{mmmfa} and noticing
\est{
Q_{2j}(0)=(\tfrac12(\tfrac12+1)\cdots (\tfrac12+j-1))^2=\frac{\Gamma(\tfrac12+j)^2}{\pi}.
}

Equation~\eqref{ftr1} follows in the same way from~\eqref{ecfd}.

\end{proof}
We can now deduce Theorem~\ref{mt}. The proof of Theorem~\ref{mtc} is analogous and Theorem~\ref{et} is a trivial corollary of Theorem~\ref{mtc}.
\begin{proof}[Proof of Theorem~\ref{mt}]
By Corollary~\ref{ctff}, we have
\es{\label{adae}
M_0^*\pr{\pm a,q}
&=\frac12(1-i)D\pr{\tfrac12,0;\pm \tfrac aq}+\tfrac12(1+i)D\pr{\tfrac12,0;\mp\tfrac aq}.\\
}
We want to apply~\eqref{ftr1}, but first we need to make a few computations. First, we observe that
\es{\label{asd1}
&\frac12(1-i)i^{- \ell}+\frac12(1+i)i^{\ell}=2^\frac12\cos\pr{\tfrac{\pi}{2}(\ell+\tfrac12)},\\
}
and, by Theorem~\ref{tff},
\es{\label{asd2}
&\frac12(1-i)D(\tfrac12+j,0,\mp\tfrac qa)i^j +\frac12(1+i)D(\tfrac12+j,0,\pm\tfrac qa)i^{-j}=\pi^{\frac12}\frac{(2\pi)^{j}}{\Gamma(\tfrac12+j)}M_j^*\pr{\mp q,a}.
}
Moreover, we have that 
\es{\label{asd3}
&-\frac12(1-i)\pr{\tfrac aq}^{\frac12}\pr{\tfrac{1\mp i}2}(\Psi(\tfrac12-j)+2\gamma-\log( 2\pi q/a)\mp\tfrac{\pi i}{2})+{}\\
&\hspace{2cm}-\frac12(1+i)\pr{\tfrac aq}^{\frac12}\pr{\tfrac{1\pm i}2}(\Psi(\tfrac12-j)+2\gamma-\log( 2\pi q/a)\pm\tfrac{\pi i}{2})=r_{\pm,j}(\tfrac aq),\\
}
where
\est{
r_{\pm,j}(z):=z^{\frac12}\times \begin{cases}
\frac{\pi }{2}& \tn{if }\pm =+,\\
(\log( 2\pi/z)-\Psi(\tfrac12-j)-2\gamma)& \tn{if }\pm =-.\\
\end{cases}
}
Finally,
\est{
&\frac12(1-i)Z_\pm\pr{0,z}+\frac12(1+i)Z_\mp\pr{0,z}=\\
&\qquad=2^\frac12\frac1{2\pi i}\int_{\pr{\frac34}}\Gamma(w) \frac{1}{\sin\pi w}\zeta\pr{\tfrac12+w}^2(2\pi z)^{-w}\cos\pr{\tfrac{\pi}{2}(w\mp\tfrac12)}dw\\
&\qquad=\frac1{2\pi i}\int_{\pr{\frac34}} \frac{\Gamma(w)}{\sin\pi w}\zeta\pr{\tfrac12+w}^2(2\pi z)^{-w}(\cos\pr{\tfrac{\pi}{2}w}\pm \sin\pr{\tfrac{\pi}{2}w})dw.\\
}
We move the line of integration to $\Re(w)=-\frac12$, passing through poles at $w=\frac12$ and at $w=0$. A quick computation shows that the residue at $w=\frac12$ is equal to $-r_{\mp,0}(1/z)$, whereas the residue at $z=0$ is
\est{
g_{\pm}(z):=\zeta(\tfrac12)^2(\tfrac{1}2\mp \tfrac12-\tfrac1\pi \log ( z/4)),
}
since  $\frac{\zeta'}\zeta(\frac12)=\frac12 ( \gamma +\pi/2+\log(8\pi))$. Thus,
\es{\label{asd4}
\frac12(1-i)Z_\pm\pr{0,z}+\frac12(1+i)Z_\mp\pr{0,z}=W_{\pm}(z)+g_{\pm}(z)-r_{\mp,0}(1/z).
}

We can now apply~\eqref{ftr1}, and by~\eqref{asd1}-\eqref{asd4} we get
\est{
M_0^*\pr{\pm a,q}&=\frac1{\pi}\sum_{j=0}^\infty \frac{(-1)^j\Gamma(\tfrac12+j)}{j!}\bigg(\pr{\frac{\pm a}{q}}^j\pi^{\frac12}M_j^*\pr{\mp q,a}+(-1)^j\pi^\frac12r_{\pm,j}(\tfrac aq)+{}\\ 
&\quad-\sum_{\ell=0}^j\frac{(-1)^{j-\ell}\Gamma(\tfrac12+j)}{(j-\ell)!}\zeta\pr{\tfrac12+\ell}^2\pr{\frac {\pm a}{2\pi q}}^{\ell}2^\frac12\cos\pr{\tfrac{\pi}{2}(\ell+\tfrac12)}\bigg)+W_{\pm}(\tfrac aq)+g_{\pm}(\tfrac aq)-r_{\pm,0}(\tfrac qa),\\
}
where we also used the reflection formula for the Gamma function. Thus, Theorem~\ref{mt} follows by the functional equation for the Riemann zeta function
\est{
\zeta(\tfrac12-\ell)=2(2\pi)^{-(\frac12+\ell)}\Gamma(\tfrac12+\ell)\cos	\pr{\tfrac{\pi}{2}(\tfrac12+\ell)}\zeta(\tfrac12+\ell)
}
and the identities
\est{
\binom{j-\frac12}{\ell-\frac12}=\frac{\Gamma(\tfrac12+j)}{(j-\ell)!\Gamma(\frac12+\ell)},\qquad \binom{j-\frac12}{j}=\frac{\Gamma(\tfrac12+j)}{\pi^\frac12\Gamma(j)}.\\
}
\end{proof}

\section{Proof of Theorem~\ref{Ypo} and of Theorem~\ref{cffd}}
Since the reciprocity formula for $M_0(a,q)$ holds only for $a,q$ prime, first we need to prove the analogue of Theorem~\ref{Ypo} for the central value of the Estermann function.

\begin{lemma}\label{lmfil}
Let $a,q>0$ with $(a,q)=1$. Let $[b_0;b_1,\cdots,b_{\kappa}]$ be the continued fraction expansion of $a/q$ and let $v_j$ be the $j$-th partial denominator. Then
 \es{\label{prefo1}
D_0\pr{\tfrac12,\pm \tfrac {a}{q}}&= \zeta\pr{\tfrac12}^2({\kappa}+1)+\frac12\sum_{j=1}^{\kappa}\pr{\frac {v_{j}}{v_{j-1}}}^{\frac12}\pr{\log \frac {v_{j}}{v_{j-1}}+\gamma -\log 8\pi -\frac {\pi }2}+{}\\
&\quad\mp\frac i2\sum_{j=1}^{\kappa}(-1)^{j}\pr{\frac {v_{j}}{v_{j-1}}}^{\frac12}\pr{\log \frac {v_{j}}{v_{j-1}}+\gamma -\log 8\pi +\frac {\pi }2}+\sum_{j=1}^{\kappa} \mathcal E\pr{\pm  (-1)^{j}\frac {v_{j-1}}{v_{j}}},
}
where $\mathcal E(x)$ is a continuous function on $\R$ satisfying  $\mathcal E\pr{x}\ll x$ for $x\ll 1$.
\end{lemma}
\begin{proof}
It is easy to see that if $\frac mq:=(-1)^{\kappa+1}\frac{\overline a}{q}$ (with $\overline a$ the reduced inverse of $a$ mod $q$), then we have $m/q=[0;c_1,\dots,c_\kappa]= [0;b_r,\dots,b_1]$. Moreover, the Euclid algorithm for $m/q$ gives
\est{
&y_1=q,\qquad y_2=m,\\
&y_{j}=c_{j}y_{j+1}+y_{j+2},\qquad j=0,\dots \kappa,\\
} 
with $y_{\kappa+1-j}=v_{j}$ (and $y_{\kappa+2}=v_{-1}=0$). 
Thus, from~\eqref{ftr}, for all $1\leq j\leq \kappa$ we have
\es{\label{fftrq}
D\pr{\frac12,\pm \frac {y_{j+1}}{y_{j}}}&=D_0\pr{\frac12,\mp \frac {y_{j+2}}{y_{j+1}}}+ \zeta\pr{\tfrac12}^2 +\frac12\pr{\frac{y_j} {y_{j+1}}}^{\frac12}  \pr{\log \pr{\frac {y_j}{y_{j+1}}}+\gamma -\log 8\pi -\frac {\pi }2}\\
&\quad\pm i\frac12\pr{\frac {y_j}{y_{j+1}}}^{\frac12} \pr{\log \frac{y_j} {y_{j+1}}+\gamma -\log 8\pi +\frac {\pi }2}+\mathcal E\pr{\pm\frac {y_{j+1}}{y_j}},\\
}
where $\mathcal E(x):=\mathcal E_0(0,x)$.
We alternate the use of~\eqref{fftrq} with the reduction modulo the denominator and we obtain
 \es{\label{prefo}
D_0\pr{\frac12,\pm \frac {m}{q}}&= \zeta\pr{\tfrac12}^2({\kappa}+1)+\frac12\sum_{j=1}^{\kappa}\pr{\frac {y_{j}}{y_{j+1}}}^{\frac12}\pr{\log \frac {y_j}{y_{j+1}}+\gamma -\log 8\pi -\frac {\pi }2}+{}\\
&\quad\mp\frac i2\sum_{j=1}^{\kappa}(-1)^j\pr{\frac {y_{j}}{y_{j+1}}}^{\frac12}\pr{\log \frac {y_j}{y_{j+1}}+\gamma -\log 8\pi +\frac {\pi }2}+\sum_{j=1}^{\kappa} \mathcal E\pr{\pm(-1)^j\frac {y_{j+1}}{y_{j}}},
}
since $D_0(\frac12,\frac01)=\zeta(\frac12)^2$. Now, by~\eqref{tffe2} (or by the functional equation for $D_0$) it follows that $D_0(\frac12,\pm\frac mq)=D_0(\frac12,\pm(-1)^{\kappa+1}\frac{ a}{q})$ and~\eqref{prefo1} then follows.
\end{proof}

\begin{corol}\label{lmfic}
Let $a,q>0$, and let $[b_0;b_1,\cdots,b_{\kappa}]$ be the continued fraction expansion of $a/q$. Then
 \est{
D_0\pr{\tfrac12,\pm \tfrac {a}{q}}&=\frac12\sum_{j=1}^\kappa b_j^{\frac12}(\log b_j +\gamma -\log 8\pi -\frac {\pi }2)\mp\frac i2\sum_{j=1}^\kappa(-1)^jb_j^{\frac12}(\log b_j +\gamma -\log 8\pi +\frac {\pi }2)+O({\kappa}).
}
\end{corol}
\begin{proof}
For $j=1,\dots \kappa$ have $v_{j}/v_{j-1}=b_j+O(1)$. Also, we have $\mathcal E\pr{x}\ll x$ for $x\ll 1$. Thus,
 \est{
D_0\pr{\tfrac12,\pm\tfrac {a}{q}}&= \zeta\pr{\tfrac12}^2({\kappa}+1)+\frac12\sum_j(b_j^{\frac12}+O(b_j^{-1/2}))(\log b_j +O(1/b_j) +\gamma -\log 8\pi -\frac {\pi }2)+{}\\
&\quad\mp\frac i2\sum_j(-1)^j(b_j^{\frac12}+O(b_j^{-1/2}))(\log b_j +O(1/b_j) +\gamma -\log 8\pi +\frac {\pi }2)+O\bigg(\sum_{j} \frac {1}{b_{j}}\bigg)\\
&= \zeta\pr{\tfrac12}^2({\kappa}+1)+\frac12\sum_jb_j^{\frac12}(\log b_j +\gamma -\log 8\pi -\frac {\pi }2)+{}\\
&\quad\mp\frac i2\sum_j(-1)^jb_j^{\frac12}(\log b_j +\gamma -\log 8\pi +\frac {\pi }2)+O\bigg(\sum_{j} \frac {\log (b_j+1)}{b_{j}^\frac12}\bigg)
}
and the Corollary follows.
\end{proof}

\begin{proof}[Proof of Theorem~\ref{Ypo} and Corollary~\ref{Ypc}]
Theorem~\ref{Ypo} follows immediately by~\eqref{adae} and~\eqref{prefo1}. Corollary~\ref{Ypc} can be then obtained from Theorem~\ref{Ypo} in the same way as in the proof of Corollary~\ref{lmfic} using the identities
\est{
\frac12M_0^*\pr{a,q}+\frac12M_0^*\pr{-a,q}
&=\frac{q^{\frac12}}{\varphi(q)}\sumstar_{\substack{\chi\mod q,\\\chi(-1)=1}} \pmd{L\pr{\frac12,\overline\chi}}^2\chi(a)+O(q^{-\frac12})
}
and
\est{
\frac12M_0^*\pr{a,q}-\frac12M_0^*\pr{-a,q}&=\frac{q^{\frac12}}{\varphi(q)}\sumstar_{\substack{\chi\mod q,\\\chi(-1)=-1}} \pmd{L\pr{\frac12,\overline\chi}}^2\chi(a)
}
which are immediate consequences of the definition of $M_0^*$.
\end{proof}
In order to prove Corollary~\ref{fccc}, we need a bound for the first moment of $M(a,q)$.
\begin{lemma}\label{lfc}
Let $q$ be prime. Then
\est{
\frac1q\sum_{a=1}^q|M(a,q)|\ll \log q.
}
\end{lemma}
\begin{proof}
The approximate functional equation for $L(\frac12,\chi)$ gives (cf. Lemma 2.3 of~\cite{You11a})
\est{
M(a,q)\ll {q^\frac12}\sum_{\substack{n\equiv \pm am\mod q,\\ (nm,q)=1}}\frac{1}{\sqrt{mn}}\min\pr{1,\pr{\frac{q}{mn}}^{100}}+{q^{-\frac12}}\sum_{\substack{ (nm,q)=1}}\frac{1}{\sqrt{mn}}\min\pr{1,\pr{\frac{q}{mn}}^{100}}
}
Thus,
\est{
\frac1q\sum_{a=1}^q|M(a,q)|&\ll\frac1{q^\frac12} \sum_{\substack{ nm\leq q}}\frac{1}{\sqrt{mn}}+\frac1{q^\frac12} \sum_{\substack{ nm\geq q}}\frac{1}{\sqrt{mn}}\pr{\frac{q}{mn}}^{100}\\
&\ll \log q
}
\end{proof}
\begin{proof}[Proof of Corollary~\ref{fccc}]
By Corollary~\ref{Ypc} we have
\est{ 
&\frac{q^{\frac12}}{\varphi(q)}\sumstar_{\substack{\chi\mod q,\\\chi(-1)=\pm1}} \pmd{L\pr{\frac12,\chi}}^2\chi(a)=\pm \frac12 f_{\pm}(a/q)+O(\log (q)),\\
}
whence, by Lemma~\ref{lfc},
\est{ 
 \frac1q \sum_{a=1}^qf_{\pm}(a/q)^2&=
4\frac{q}{\varphi(q)^2}\frac1q \sum_{a=1}^q\bigg(\sumstar_{\substack{\chi\mod q,\\\chi(-1)=\pm1}} \pmd{L\pr{\frac12,\chi}}^2\chi(a)\bigg)^2\\
&\quad+O\bigg(\log q \frac1q \sum_{a=1}^q \bigg|\frac{q^{\frac12}}{\varphi(q)}\sumstar_{\substack{\chi\mod q,\\\chi(-1)=\pm1}} \pmd{L\pr{\frac12,\chi}}^2\chi(a)\bigg|\bigg)+O(\log^2 q),\\
&=
4\frac{1}{\varphi(q)} \sumstar_{\substack{\chi\mod q,\\\chi(-1)=\pm1}} \pmd{L\pr{\frac12,\chi}}^4+O\pr{\log^2q}\\
}
and the corollary follows by Young's theorem~\cite{You11b}.
\end{proof}
\begin{proof}[Proof of Theorem~\ref{cffd}]
Let $x,y\in\R$. For any $\eps>0$, take $b_0,b_1,\cdots,b_r$ be such that $|x-w|<\eps$, where $w:=[b_0;b_1,\cdots,b_r]$ and $r\geq5$ is odd. For $b_{r+1},b_{r+2}\in\Z_{>0}$, let $a/q:=[0;b_1,\cdots,b_r,b_{r+1},b_{r+2}]$ and let $v_j$ be the $j$-th partial denominator. We will consider $x$ and $w$ (and thus $r$, $b_1,\dots,b_r$, $v_1,\dots, v_r$) to be fixed and $b_{r+1},b_{r+2}$ to be large variables that we will choose at the end of the argument.
Using Corollary~\ref{lmfic}, we obtain
\est{
\eta\pr{\tfrac aq}&=\sum_{\substack{j=1,\\j\tn{ odd}}}^{r+2}\pr{\frac{v_{j}}{v_{j-1}}}^{\frac12}(\log \pr{{v_{j}}/{v_{j-1}}} +\gamma -\log 8\pi)-\frac\pi 2\sum_{\substack{j=1,\\j\tn{ even}}}^{r+2}\pr{\frac{v_{j}}{v_{j-1}}}^{\frac12}+\zeta(\tfrac 12)^2(r+2)+{}\\
&\quad+\sum_{j=1}^{r+2}\pr{(1-i)\mathcal{E}\pr{(-1)^j \frac{v_{j-1}}{v_j}}+(1+i)\mathcal{E}\pr{(-1)^{j+1}\frac{v_{j-1}}{v_j}}}.
}
Now, $v_{j+1}/v_{j}=b_{j+1}+\frac{v_{j-1}}{v_j}=b_{j+1}+O(1)$. Thus, since $\mathcal{E}(z)\ll z$ for $z\ll1$ we have
\est{
\eta\pr{\tfrac aq}&=c_{x,w}+\pr{{v_{r+2}}/{v_{r+1}}}^{\frac12}(\log \pr{{v_{r+2}}/{v_{r+1}}} +\gamma -\log 8\pi)-\frac\pi 2\pr{{v_{r+1}}/{v_{r}}}^{\frac12}+o(1)\\
&=c_{x,w}+b_{r+2}^{\frac12}(\log b_{r+2}+\gamma -\log 8\pi)-\frac\pi 2b_{r+1}^{\frac12}+o(1),
}
as $b_{r+1}$ and $b_{r+2}$ tend to infinity, where $c_{x,w}$ depends only on $x$ and $w$.
Now, taking
\est{
b_{r+1}=\left[\frac{4}{\pi^2}\pr{c_{x,w}+b_{r+2}^{\frac12}(\log b_{r+2}+\gamma -\log 8\pi)-y}^2\right],
}
we have 
\est{
\eta\pr{\tfrac aq}=y+o(1).
}
as $b_{r+2}$ goes to infinity and the Theorem then follows.
\end{proof}

\section{Moments with two twists}

In this section we prove Corollary~\ref{csb} and Corollary~\ref{c3t}. First we prove the following Lemma which relates the continued fraction expansion of $\{\frac{h\overline k}{q}\}$ to that of $\{-\frac{h\overline q}{k}\}$ and $\{-\frac{k\overline q}{h}\}$.

\begin{lemma}\label{hat}
Let $h,k,q\in\Z_{>0}$, with $(h,k)=(h,q)=(k,q)=1$. Let $[0;b_1,\dots,b_r]$ and $[0;c_1,\dots,c_{s}]$ be the continued fraction expansion of $\{-\frac{h\overline q}{k}\}$ and $\{-\frac{k\overline q}{h}\}$ respectively with $r$ and $s$ even (where $\overline q$ denotes the inverse modulo the denominator). If $q\geq 4hk$, then
 the continued fraction expansion of $\{\frac{h\overline k}{q}\}$ is $[0;b_1,\dots,b_r, \frac{q}{hk}+O(1),c_s,\dots,c_1]$.
\end{lemma}
\begin{proof}
First, we observe that if $q\geq 2hk$ then $\{-\frac{h\overline q}{k}\}$ is an even convergent for the continued fraction of $\{\frac{h\overline k}{q}\}$. Indeed, this is equivalent to showing that $\{-\frac{h\overline q}{k}\}<\{\frac{h\overline k}{q}\}$ and that for every fraction $\frac{u'}{k'}$ with $0<k'<k$ we have
\est{
|xk'-u'|>|xk-u|,
}
where $x:=\{\frac{h\overline k}{q}\}$ and $\frac{u}k:=\{-\frac{h\overline q}{k}\}$. The fact that  $\{-\frac{h\overline q}{k}\}<\{\frac{h\overline k}{q}\}$ follows immediately from 
\est{
\frac{h\overline k}{q}\equiv -\frac{h\overline q}{k}+\frac{h}{qk}\mod 1,
}
which for $|h|<q$ implies
\es{\label{rcf}
\pg{\frac{h\overline k}{q}}=\pg{ -\frac{h\overline q}{k}}+\frac{h}{qk}.
}
Now, if $|xk'-u'|\leq |xk-u|$ for some $u',k'$ such that $0<k'<k$, then
\est{
\frac1{kk'}\leq \pmd{\frac{u'}{k'}-\frac{u}{k}}\leq\pmd{x-\frac uk}+\pmd{x-\frac {u'}{k'}}\leq \pr{1+\frac{k}{k'}}\pmd{x-\frac uk}<2\frac{k}{k'}\pmd{x-\frac uk}=\frac{2h}{qk'},
}
by~\eqref{rcf}, and so we obtain a contradiction.

Thus, we have that $\{\frac{h\overline k}{q}\}=[0;d_1,\dots,d_n]=[0;b_1,\dots,b_r,d_{r+1},\dots,d_n]$ for some odd integer $n>r$ and some $d_{1},\dots,d_n\in\Z_{>0}$. Moreover, 
\est{
\pg{\frac{h\overline k}{q}}=[0;d_1,\dots,d_{n}]=[0;b_1,\dots,b_r,y]=\frac{y\alpha_r+\alpha_{r-1}}{y\beta_r+\beta_{r-1}},
}
where $y=[d_{r+1};d_{r+2},\dots,d_n]$ and $\frac{\alpha_i}{\beta_i}$ denotes the $i$-th convergent of $\{-\frac{h\overline q}{k}\}$. Therefore, by~\eqref{rcf} we have
\est{
\frac{h}{qk}=\pg{\frac{h\overline k}{q}}-\pg{ -\frac{h\overline q}{k}}=\frac{y\alpha_r+\alpha_{r-1}}{y\beta_r+\beta_{r-1}}-\frac{\alpha_r}{\beta_r}
=\frac{1}{\beta_r}\pr{\frac{1}{\beta_ry+\beta_{r-1}}},
}
since $\beta_r\alpha_{r-1}-\beta_{r-1}\alpha_{r}=(-1)^r=1$. It follows that 
$
y=\frac q{hk}-\frac{\beta_{r-1}}{k}
$
and thus 
\est{
d_{r+1}=\frac{q}{hk}-\delta_{h,k,q}
}
with $0<\delta_{h,k,q}=[0;d_{r+2},\dots,d_n]+\frac{\beta_{r-1}}{k}<2$.

Now, we have that the continued fraction expansion of  $\{\frac{k\overline h}{q}\}$ is $[0;d_n,\dots,d_{1}]$ and repeating the same calculations as above we see that $[0;d_n,\dots,d_{1}]=[0;c_1,\dots,c_s,\frac{q}{hk}-\delta_{h,k,q}',d_{n-s-1}\dots,d_{1}]$, with $0<\delta_{h,k,q}'<2$.  Thus, to conclude we just need to show that $n=r+s+1$ if $q\geq 4hk$. First we observe that $n\leq r+s+1$, since otherwise, we would have
\est{
\frac{h\overline k}{q}=[0;b_1,\dots,b_{r},\frac{q}{hk}-\delta_{h,k,q},d_{r+2},\dots,d_{n-s-1},\frac{q}{hk}-\delta_{h,k,q}',c_s,\dots,c_1]
}
and thus comparing the denominators we obtain $q>h\pr{\frac {q}{hk}-2}^2k\geq q$ (for $q\geq 4hk$), which gives a contradiction. Similarly, if we had $n<r+s-1$, then we would get that $q$ is less or equal than the denominator of $[0;b_1,\dots,b_r,c_s,\dots,c_1]$, which is bounded by $hk$. This again gives a contradiction and so the proof of the Lemma is complete.
\end{proof}

\begin{proof}[Proof of Corollary~\ref{csb} and Corollary~\ref{c3t}]
Let $[0;b_1,\dots,b_r]$ and $[0;c_1,\dots,c_{s}]$ be the continued fraction expansion of $\{-\frac{h\overline q}{k}\}$ and $\{-\frac{k\overline q}{h}\}$ respectively, with $r$ and $s$ even. By Corollary~\ref{Ypc} and Lemma~\ref{hat}, we have
\est{
M_\pm(h,k;q)&= \frac12\pr{\frac q{hk}}^{\frac12}\pr{\log \frac q{hk} +\gamma -\log 8\pi \mp\frac {\pi }2}+O(\log q)+{}\\
&\quad\pm \frac12\sum_{j=1}^r(\pm 1)^jb_j^{\frac12}(\log b_j +\gamma -\log 8\pi \mp\frac {\pi }2)\pm \frac12\sum_{j=1}^s(\pm 1)^jc_j^{\frac12}(\log c_j +\gamma -\log 8\pi \mp\frac {\pi }2).
}
Thus, to obtain~\eqref{csbf} it is enough to bound trivially the two series on the second line. Moreover, if $h$ and $k$ are both primes, we obtain Corollary~\ref{c3t}, since in this case the previous formula becomes
\est{
M_\pm(h,k;q)&= \frac12\pr{\frac q{hk}}^{\frac12}\pr{\log \frac q{hk} +\gamma -\log 8\pi \mp\frac {\pi }2}+M_\pm(-h,q;k)+M_\pm(-k,q;h)+O(\log q)\\
&= \frac12\pr{\frac q{hk}}^{\frac12}\pr{\log \frac q{hk} +\gamma -\log 8\pi \mp\frac {\pi }2}\pm M_\pm(h,q;k)\pm M_\pm(k,q;h)+O(\log q).
}
\end{proof}

\appendix

\end{document}